\documentclass[a4paper,11pt]{article}
\RequirePackage[OT1]{fontenc }
\RequirePackage{amsthm,amsmath}

\usepackage{colortbl}
\usepackage{multirow}
\usepackage{fullpage}
\usepackage{amsfonts,amssymb}
\usepackage{graphicx}
\usepackage{enumerate}
\usepackage{bm}
\usepackage{natbib}
\usepackage[titletoc,title]{appendix}

\usepackage{tikz}
\usepackage{pgf}
\usepackage{caption}
\usepackage{subcaption}
\usepackage{pgfplots}
\usepackage{float}
\usepackage{wrapfig}
\usepackage{xcolor} 
\usepackage{hyperref}
\usetikzlibrary{spy,calc}

\providecommand{\keywords}[1]
{
  \small	
  \textbf{\textit{Keywords---}} #1
}

\newtheorem{theorem}{Theorem}
\newtheorem{corollary}{Corollary}
\newtheorem{proposition}[theorem]{Proposition}
\newtheorem{lemma}[theorem]{Lemma}

\newtheorem{assumption}{Assumption}

\renewcommand{\le}{\leqslant}
\renewcommand{\ge}{\geqslant}
\renewcommand{\leq}{\leqslant}
\renewcommand{\geq}{\geqslant}
\renewcommand{\tilde}{\widetilde}

\DeclareMathOperator{\var}{Var}

\DeclareMathOperator{\argmin}{argmin}
\DeclareMathOperator{\tr}{tr}

\title{High dimensional regression for regenerative time-series: an application to road traffic modeling} 
\author{Mohammed Bouchouia and Fran\c cois Portier  \\\\
T\'el\'ecom Paris, Institut Polytechnique de Paris}
\date{}

\begin{document}
\maketitle
    \begin{abstract}

      A statistical predictive model in which a high-dimensional time-series regenerates at the end of each day is used to model road traffic.
      Due to the regeneration, prediction is based on a daily modeling using a vector autoregressive model that combines linearly the past observations of the day. Due to the high-dimension, the learning algorithm follows from an $\ell_1$-penalization of the regression coefficients. Excess risk bounds are established under the high-dimensional framework in which the number of road sections goes to infinity with the number of observed days.  Considering floating car data observed in an urban area, the approach is compared to state-of-the-art methods including neural networks. In addition of being highly competitive in terms of prediction, it enables the identification of the most determinant sections of the road network.

    \end{abstract}
    

\section{Introduction}


As the world witnesses the negative effects of traffic congestion, including pollution and economically ineffective transportation \citep{bull2003,pellicer2013global,harrison2011theory}, achieving smart mobility has become one of the leading challenges of emerging cities \citep{washburn2009helping}. In order to set up effective solutions, such as developing intelligent transportation management systems for urban planners, or extending the road network efficiently, smart cities first need to accurately understand road traffic. 
This paper investigates interpretable predictive models estimated from \textit{floating car data}, allowing for the identification of road traffic determinants.


The proposed methodology addresses two important issues regarding \textit{floating car data}. First, in contrast with a traditional \textit{time series} analysis \citep{brockwell+d+f:1991}, the number of vehicles using the network almost vanishes during the night. Hence, in terms of probabilistic dependency, the road traffic between the different days is assumed independent. Such a phenomenon is referred to as ``\textit{regeneration}'' in the Markov chains literature \citep{meyn+t:2012} where the road network ``regenerates'' at the beginning of each new day. Second, the size of the road network, especially in urban areas, can be relatively large compared to the number of observed days. This implies that the algorithms employed must be robust to the well-known \textit{high-dimensional regression} \citep{giraud2014introduction} setting in which the features are numerous.  

For each time $t\in \{0,\ldots,  T\}$ of day $ i \in \{1,\ldots,  n\}$, we denote the vector of speeds registered in the road network by $W_{t}^{(i)}\in \mathbb R ^{p}$. 

 Hence $p$, $n$ and $T+1$ stand for the number of sections in the network, the number of days in the study, and the number of time instants within each day.  
 Inspired by the \textit{time series analysis}, the proposed model is similar to the popular \textit{vector auto-regressive} (VAR) model, as described in the econometric literature \citep{hamilton:1994}. The one difference is that it only applies within each new day due to the regeneration property.

We therefore consider the following linear regression model, called regenerative VAR,
\begin{align}\label{modeling}
L(W_{t}^{(i)}) =   b_t + A W_t^{(i)} ,\qquad   \text{for each }    t\in \{0,\ldots, T-1\},\ i \in \{1,\ldots , n\},
\end{align}
with $A\in \mathbb R^{p\times p}$ and $b_t\in \mathbb R^p$. This model is used to predict the next value $W_{t+1}^{(i)} $. The parameter $A$ encodes for the influence between different road sections, and the parameters $(b_t )_{t=0,\ldots ,T-1}$ account for the daily (seasonal) variations.

The approach taken for estimating the parameters of the regenerative VAR, $(b_t )_{t=0,\ldots ,T-1}$ and $A$, follows from applying ordinary least-squares (OLS) while penalizing the coefficients of $A$ using the $\ell_1$-norm just as in the popular LASSO procedure \citep{tibshirani:1996}; see also \cite{buhlmann2011statistics,giraud2014introduction,hastie+t+w:2015} for reference textbooks.
The estimator is computed by minimizing over $(b_t )_{t=0,\ldots ,T-1}\in \mathbb R ^{p\times T}$ and $A\in \mathbb R ^{p\times p}$ the following objective function
\begin{align}\label{objective}
\sum_{i=1}^n\sum_{t=0}^{T-1}  \| W^{(i)}_{t+1} - b_t - A W^{(i)}_{t} \|_F^2 + 2\lambda \|A\|_1 ,
\end{align}
where $\|A\|_1 = \sum_{  1\leq k,\ell \leq p }   |A_{k,\ell} |$ and $\|\cdot \|_F$ stands for the Frobenius norm.
While the estimation of standard VAR models (without regeneration) has been well-documented for decades \citep{brockwell+d+f:1991}, only recently has penalization been introduced to estimate the model coefficients; see, among others, \citep{valdes+p+s+j+l+a+v+b+m+l+c:2005,wang+l+t:2007,haufe+m+k+n+k:2010,song+b:2011,michailidis2013autoregressive,kock+c:2015,basu+m:2015,baek+d+p:2017}. The aforementioned references advocate for the use of the LASSO or some of its variants in time-series prediction when the dimension of the time series is relatively large. 
Other variable selection approaches in VAR models, but without using the LASSO, are proposed in \cite{davis+z+z:2016}.  

The regeneration property differentiates the mathematical tools used in this paper from the time-series literature mentioned above. The independence between the days allows us to rely on standard results dealing with sums of independent random variables and therefore to build upon results dealing with the LASSO based on independent data \citep{bickel+r+t:2009}. 
From an estimation point of view, two key aspects are related to the regenerative VAR model: (i) the regression output is a high-dimensional vector of size $p$ and (ii) the model is linear with random covariates. This last point makes our study related to random design analyses of the LASSO \citep{bunea+t+w:2007,van+b:2009}. Another issue with regards to the regenerative VAR model is that of model switching, which occurs when the use of two different matrices $A$, each within two different time periods, improve the prediction. 
This kind of variation in the predictive model might simply be caused by changes in the road traffic intensity. Therefore, a key objective with regards to the regenerative VAR model is detecting such change points based on observed data.

From a theoretical standpoint, we adopt an asymptotic framework which captures the nature of usual road traffic data in which $p$ and $n$ are growing to infinity ($p$ might be larger than $n$) whereas $T$ is considered to be fixed.
We first establish a bound on the predictive risk, defined as the normalized prediction error, for the case $\lambda = 0$. This situation corresponds to \textit{ordinary least-squares}. The order is in $p / n$ and therefore deteriorates when the value of $p$ increases. 
Moreover, the conditions for its validity are rather strong because they require the smallest eigenvalue of the Gram matrix to be large enough. Then, under a sparsity assumption on the matrix $A$, claiming that each line of $A$ has a small number of a non-zero coefficient (each section is predicted based on a small number of sections), we study the regularized case when $\lambda >0$. In this case, even when $p$ is much larger than $n$, we obtain a bound of order $1/n$ (up to a logarithmic factor), and the eigenvalue condition is alleviated as it concerns only the eigenvalues restricted to the \textit{active} variables. Finally, these results are used to demonstrate the consistency of a cross-validation change point detection procedure under regime switching.

From a practical perspective, the regenerative VAR, by virtue of its simplicity, contrasts with past approaches, mostly based on \textit{deep neural networks}, that have been used to handle road traffic data. For instance, in \citep{lv2014traffic}, \textit{autoencoders} are used to extract spatial and temporal information from the data before predictions. In \citep{dai2017deeptrend}, \textit{multilayer perceptrons} (MLP) and \textit{long short-term memory} (LSTM)  are combined together to analyze time-variant trends in the data. Finally, LSTM with spatial graph convolution are employed in \citep{lv2018lc,li2017diffusion}.  Three advantages of the proposed approach are the following:

\begin{enumerate}[(i)]
\item It allows for the consideration of very large road networks which include not only main roads (major highways), which is typically the case, but also primary roads. 
The previous neural network methods either use data collected from fixed sensors on the main roads \citep{dai2017deeptrend,lv2014traffic} or, when using floating car data, restrict the network to the main roads ignoring primary roads \citep{epelbaum+g+l+m:2017}. 
Even if this avoids overfitting, as it leaves out a large number of features compared to the sample size, there might be some loss of information in reducing the data to an arbitrary subset. 
\item  The estimated coefficients are easily interpretable thanks to the linear model and the LASSO selection procedure, which shrinks irrelevant sections to zero. This provides data-driven graphical representations of the dependency within the network that could be useful for road maintenance. Once again, this is in contrast with complex deep learning models in which interpretation is known to be difficult. 
\item  Changes in the distribution of road traffic during the day can be handled easily using a regime switching approach; this consists of a simple extension of the initial regenerative VAR proposed in \eqref{modeling} in which the matrix $A$ is allowed to change over time. 
\end{enumerate}

To demonstrate the practical interest of the proposal, the data used is concerned with the urban area of a French city, Rennes, made of $n = 144$ days, $p = 556$ road sections and $T +1 = 20$ time instants (from $3$pm to $8$pm) within each day (see Figure \ref{fig:dataset_snap_shot}). Among all the considered methods, including the classical baseline from the time series analysis as well as the most recent neural network architecture, this is the regime switching model that yields the best performance.

The outline is as follows. In Section \ref{sec:framework}, the probabilistic framework is introduced. The optimal linear predictor is characterized, and the main assumptions are discussed. In Section \ref{sec:results}, we present the main theoretical results of the paper that are bounds on the prediction error of \eqref{objective}. Section \ref{sec:model_switch} investigates the regime switching variant. A comparative study including different methods applied to the real data presented before is proposed in Section \ref{sec:simu}. A simulation study is conducted in Section \ref{sec:simu2}. All the proofs of the stated results are gathered in the Appendix.

\begin{figure}
  \centering
  \begin{subfigure}[b]{0.49\textwidth}\centering
  \includegraphics[width=\textwidth]{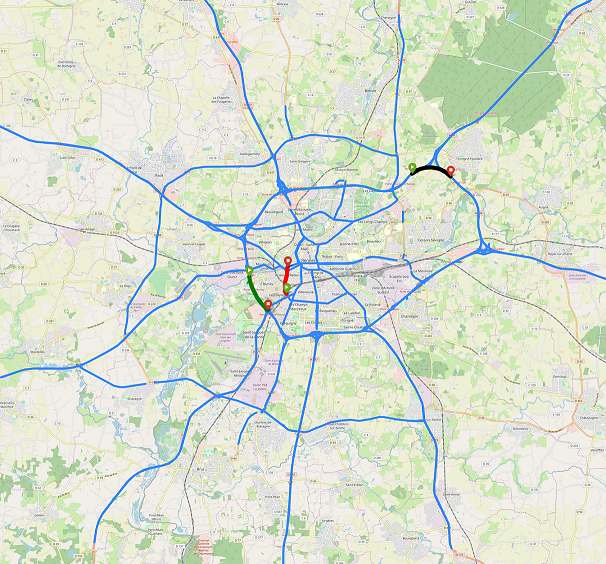}
  \caption{}
  \end{subfigure}
  \begin{subfigure}[b]{0.49\textwidth}\centering\vfill
    \begin{subfigure}[b]{0.43\textwidth}
      \includegraphics[width=1\textwidth]{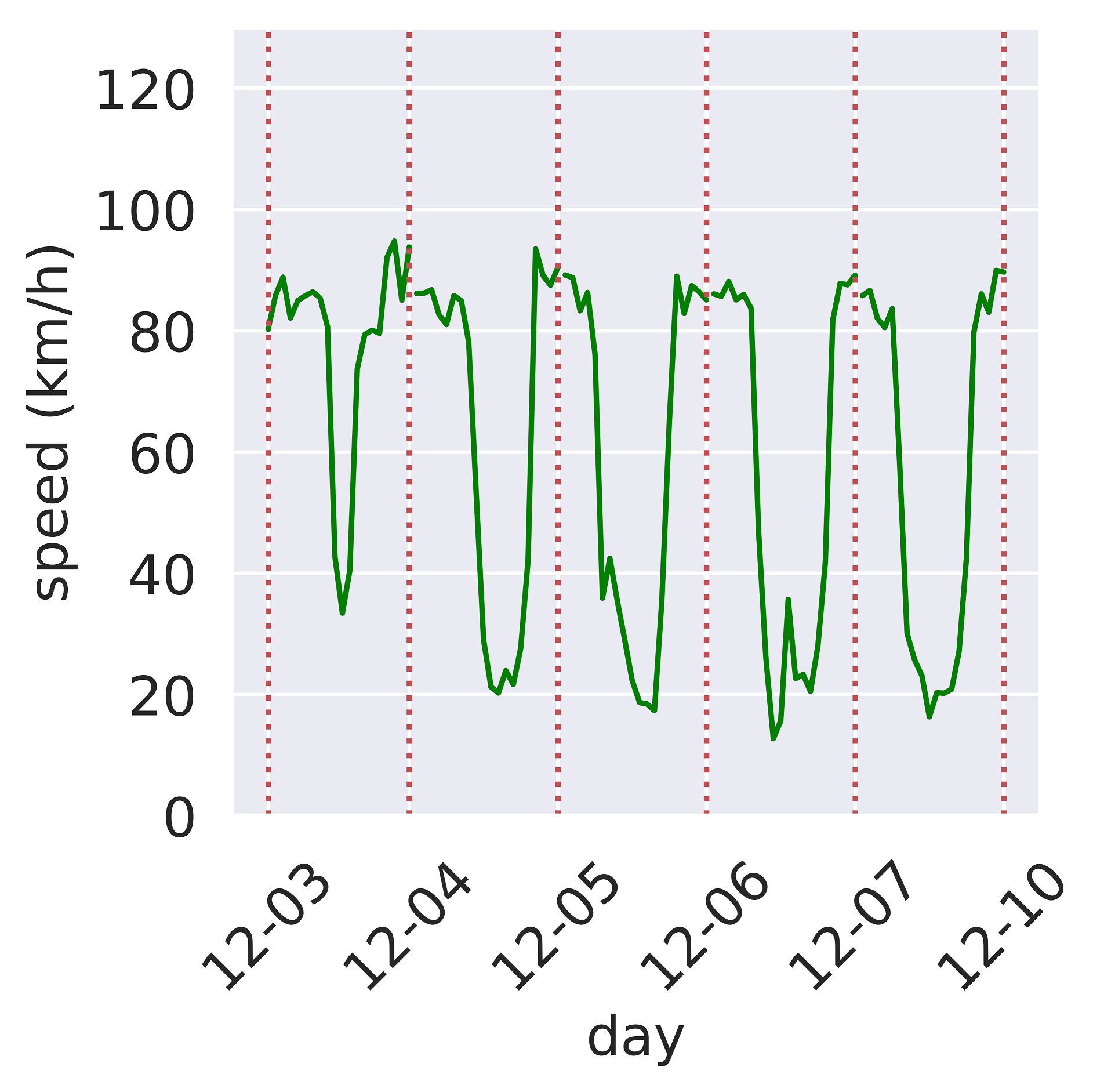}
      \caption{}
    \end{subfigure}
    \hfill
    \begin{subfigure}[b]{0.43\textwidth}
      \includegraphics[width=1\textwidth]{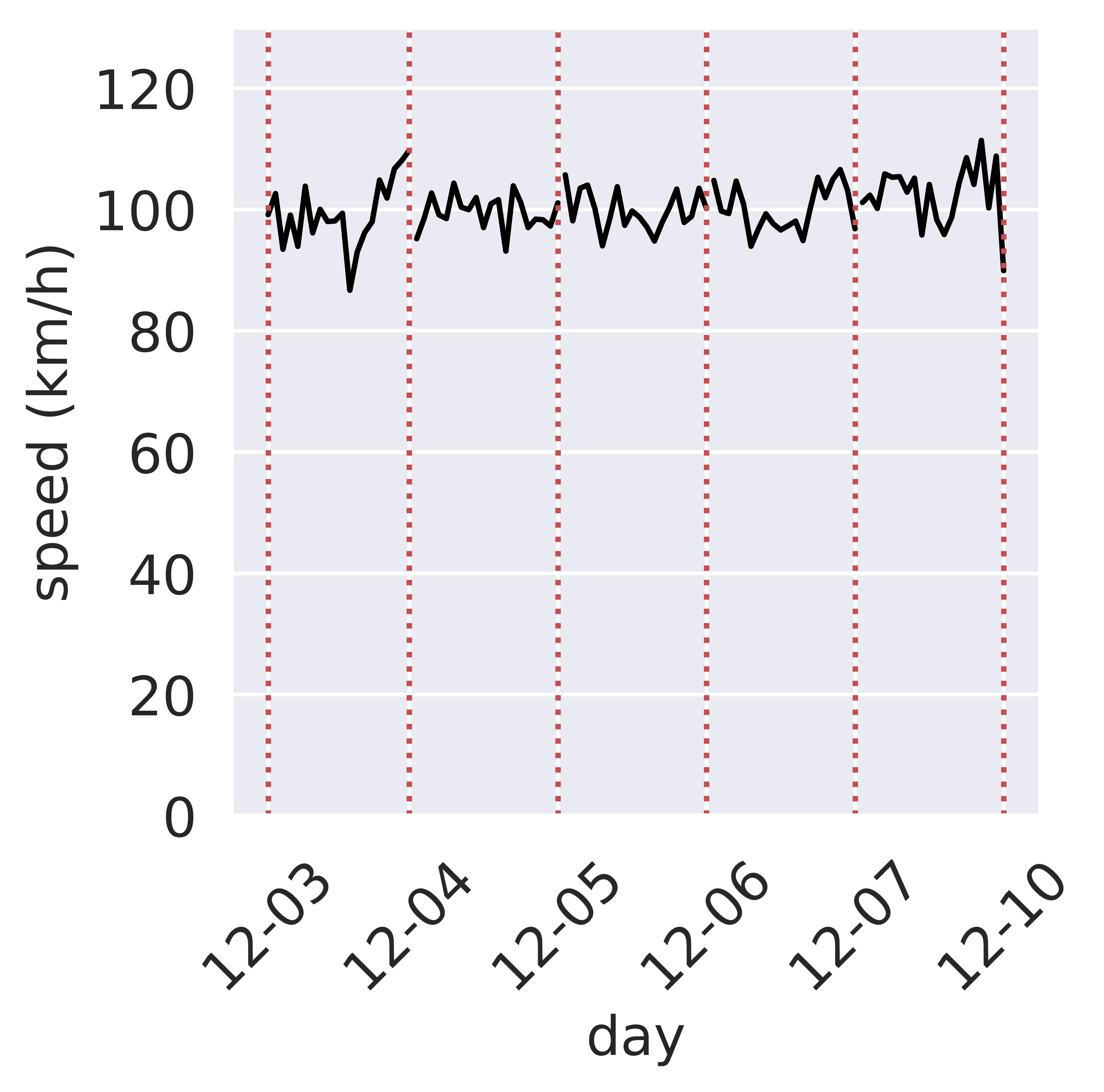}
      \caption{}
    \end{subfigure}
    \vfill
    \begin{subfigure}[b]{0.43\textwidth}
      \includegraphics[width=1\textwidth]{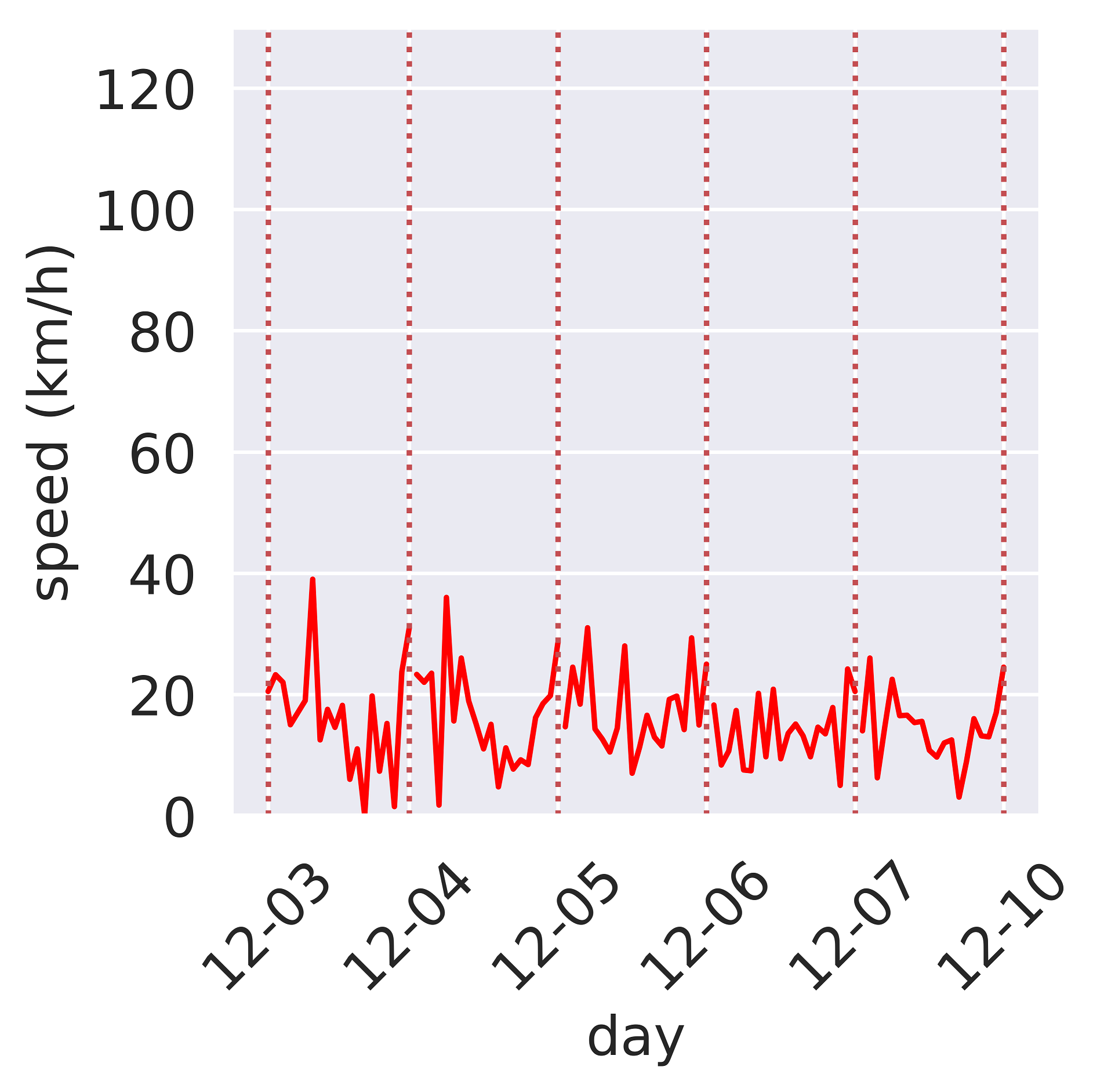}
      \caption{}
    \end{subfigure}
    \hfill
    \begin{subfigure}[b]{0.43\textwidth}
      \includegraphics[width=1\textwidth]{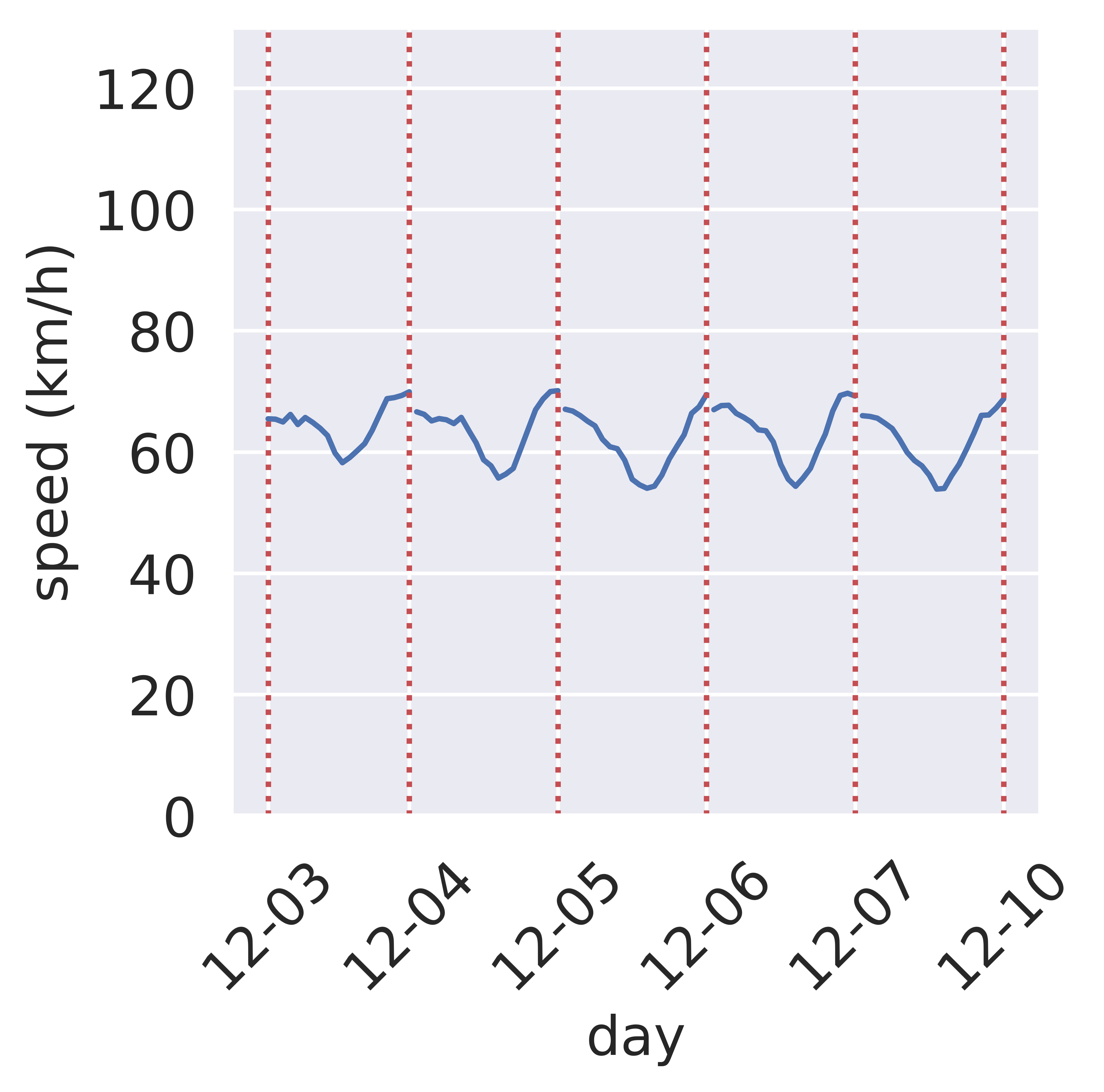}
      \caption{}
    \end{subfigure}
  \end{subfigure}
\caption{(a) Rennes road network is colored in blue and $3$ representative sections are colored in green, black and red; (b) speeds observed during $6$ days of December $2018$ in the green section (cyclic behavior); (c) in the black section (high speed section); (d) in the red section (low speed section); (e) average speed over all sections in the road network during the same $6$ days of December $2018$.}
\label{fig:dataset_snap_shot}
\end{figure}




%
%
%
%

\section{Probabilistic framework}\label{sec:framework}

Let $ \mathcal S = \{1,\ldots, p\} $ denote an index set of $p\geq 1$ real-valued variables that evolve during a time period $\mathcal T=\{0,\ldots, T\}$. These variables are gathered in a matrix of size $p\times (T+1)$ denoted by $W = (W_{k,t})_{k\in \mathcal S,\, t\in \mathcal T }  $. In road traffic data, $k\in \mathcal S$ stands for the section index of the road network $\mathcal S$, $t\in \mathcal T$ for the time instants within the day, and $W_{k,t}$ is the speed recorded at section $k$ and time $t$. We denote by $W_{\{\cdot,t\}} \in \mathbb R^{p}$ the vector of speed at time $t\in \mathcal T$. The following probabilistic framework will be adopted throughout the paper.

\begin{assumption}[probabilistic framework]\label{ass_1}
Let $(\Omega,\mathcal F,P)$ be a probability space. The matrix $W = (W_{\{\cdot,0\}},\ldots, W_{\{\cdot,T\}}) \in \mathbb R^{p\times (T+1)}$ is a random element valued in $\mathbb R^{p\times (T+1)} $ with distribution $P$. For all $k\in \mathcal S  $ and $t \in \mathcal T $, we have $E [W_{k,t}^2]<+\infty$  ($E$ stands for the expectation with respect to $P$). 
\end{assumption}

The task of interest is to predict the state variable at time $t$, $W_{\{\cdot,t\}}$, using the information available at time $t-1$, $W_{\{\cdot,t-1\}}$.  The model at use to predict $W_{\{\cdot,t\}}$ is given by the mapping $ L(W_{\{\cdot,t-1\}} ) = b_t + A W_{\{\cdot,t-1\}}$, where $A\in \mathbb R^{p\times p}$ and $b_t \in \mathbb R^{p}$ are parameters that shall be estimated. Define
\begin{align*}
&Y =  (W_{\{\cdot,1\}},\ldots, W_{\{\cdot,T\}}) \qquad \text{and} \qquad X = (W_{\{\cdot,0\}},\ldots, W_{\{\cdot,T-1\}}).
\end{align*}
 Equipped with this notation, the problem can be expressed as a simple matrix-regression problem with covariate $X\in \mathbb R^{p\times T} $ and output $Y\in \mathbb R^{p\times T} $.
In matrix notation, the model simply writes
\begin{align}\label{eq_linearmodel_matrix}
\mathcal L = \{ L( X ) = b + A X \,:\, b \in \mathbb R^{p\times T},\, A\in \mathbb R^{p\times p} \},
\end{align}
with $ L(X_{\{\cdot,1\}},\ldots, X_{\{\cdot,T\}}) = (L(X_{\{\cdot,1\}}),\ldots, L(X_{\{\cdot,T\}}))$. The number of parameters needed to describe a single element in $\mathcal L$ is $p (p+T)$. Because the matrix $A$ is fixed over the period $\mathcal T$, this model reflects a certain structural belief on the dependency along time between the state variables $W_{\{\cdot,t\}}$. Each element in $\mathcal L$ is defined by a baseline matrix $b\in \mathbb R^{p\times T}$ which carries out the average behavior in the whole network and a matrix  $A \in \mathbb R^{p\times p}$ which encodes for the influence between the different road sections. Alternative models are discussed at the end of the section. The accuracy of a given model $L\in {\mathcal {L}}$ is measured through the (normalized) $L_2(P)$-risk given by
\begin{align*}
R(L) =   p^{-1}  E [\| Y- L(X) \|_F^2 ]  .
\end{align*}
Note that, because of the factor $1/p$, $R(L)$ is just the averaged squared error of $L$. Among the class $\mathcal L$, there is an optimal predictor characterized by the usual normal equations. This is the statement of the next proposition.


\begin{proposition}\label{prop:optimal_matrix}
\label{prop:uniq_min_homo}
Suppose that Assumption \ref{ass_1} is fulfilled. There is a unique minimizer $L^\star \in \argmin_{L\in\mathcal L} R(L)$. Moreover $L = L^\star$ if and only if
$   E [Y - L  (X) ]= 0 $ and $  E[ (  Y-  L  (X) ) X^T ] = 0$.
 \end{proposition}

%
%

The following decomposition of the risk underlines the prediction loss associated to the predictor $L$ by comparing it with the best predictor $L^\star  $ in $\mathcal L$. Define the excess risk by
\begin{align*}
\mathcal E (L) =  R(L) - R(L^\star) , \qquad  L\in  {\mathcal L}.
\end{align*}


\begin{proposition}\label{prop:risk_decomp}
Suppose that Assumption \ref{ass_1} is fulfilled. It holds:
\begin{align*}
\forall L\in  {\mathcal L},\qquad & \mathcal E (L) = p^{-1}  E [\| L(X) - L^\star(X)  \|_F^2].
\end{align*}
\end{proposition}

In what follows, bounds are established on the excess risk under a certain asymptotic regime where both $n $ and $p := p_n$ go to infinity whereas $T$ remains fixed. Though theoretic, this regime is of practical interest as it permits to analyze cases where the number of sections  $p$ is relatively large (possibly greater than $n$). This implies that the sequence of random variables of interest must be introduced as a triangular array. For each $n\geq 1$, we observe $n$ realizations of $W$: $W^{(1)},\ldots, W^{(n)}$ where for each $i$, $W^{(i)}\in \mathbb R^{p\times (T+1)} $. The following modeling assumption will be the basis to derive theoretical guarantees  on  $\mathcal E (L) $.

\begin{assumption}[daily regeneration] \label{ass_2}
For each $n\geq 1$, $(W^{(i)})_{1\le i\le n}$ is an independent and identically distributed collection of random variables defined on $(\Omega, \mathcal F, P)$.
\end{assumption}

The daily regeneration assumption has two components: the independence between the days and the probability distribution of each day which stays the same. 
The independence is in accordance with the practical use of the road network where at night only a few people use the network so it can regenerate and hence ``forgets'' its past. 
The fact that each day has the same distribution means essentially that only days of similar types are gathered in the data.
For instance, this assumption might not hold when mixing weekend days and workdays. This assumption also implies that the network structure remains unchanged during data collection.

\paragraph{Empirical check of the  daily regeneration assumption}

 

To check the daily regeneration assumption on the real data presented in the introduction, we extract a set of $p(T+1) = 11120$ one-dimensional time series. Each of these is associated to a specific instant in the day $t$ and a specific section $k$, i.e., each time series is given by $(W_{k,t}^{(i)})_{i=1,\ldots, n}$ using the introduced notation. For each series, we run the test presented in \cite{broock1996test} where the null hypothesis is that the time series is independent and identically distributed. 
For most sections, $88\%$, the null hypothesis is not rejected at a nominal level of $95\%$. This clearly validates the daily regeneration assumption.
The share of sections, $12\%$, for which the hypothesis has been rejected, can be explained by the presence of missing values in the data which have been replaced by the average value (see Section \ref{sec:simu} for details). It is also likely that external events such as roadworks and constructions have impacted the result of the test.

\paragraph{Alternative models}

Two alternative modeling approaches might have been considered at the price of additional notation and minor changes in the proofs of the results presented in the next section.
 The first alternative is obtained by imposing the matrix $A$ to be diagonal. In this case, each $1$-dimensional vector coordinates is fitted using an auto-regressive model based on one single lag. The second alternative is to use more than one lag, say $H\geq 1$, to predict the next coming instance. This is done by enlarging the matrix $A$ to $(A_1,\ldots, A_H)$ where each $A_h \in \mathbb R ^{p\times p}$ and by stacking in $X$ the $H$ previous lags. These two variations, though interesting, are not presented in the paper for the sake of readability. Another approach which will be addressed in Section \ref{sec:model_switch} is when the matrix $A$ is allowed to change across time. 

\section{Empirical risk minimization}\label{sec:results}

\subsection{Definitions and first results}\label{common_decomposition}

Given a sequence $(W^{(i)})_{1\leq i\leq n}\subset \mathbb R^{p\times (T+1)}$ and a regularization parameter $\lambda\geq 0$, define the estimate 
\begin{align}\label{initial_pb}
(\hat b, \hat A) \in \argmin _{b \in \mathbb R^{p\times T}, \, A\in \mathbb R^{p\times p}} 
\left(\sum_{i=1}^n  \| Y^{(i)} - b - A  X^{(i)} \|_F^2\right) + 2\lambda \|A\|_1 ,
\end{align}
with 
\begin{align*}
 Y^{(i)}  = (W_{\{\cdot,1\}}^{(i)},\ldots, W_{\{\cdot,T\}}^{(i)})  \qquad \text{and} \qquad  X^{(i)}  = (W_{\{\cdot,0\}}^{(i)},\ldots, W_{\{\cdot,T-1\}}^{(i)}).
\end{align*}
For any vector $u \in \mathbb R^{p}  $ and any integer $q\geq 1$, the $\ell_q$-norm is defined as $\|u\|_q^q  = \sum_ {k=1}^p |u_{k} |^q $. 
The prediction at point $X\in \mathbb R^{p\times T}$ is given by
$$\hat L (X) = \hat b  + \hat A X.$$
As the intercept in classical regression, the matrix parameter $\hat b$ is only a centering term. Indeed, when minimizing \eqref{initial_pb} with respect to $b$ only, we find $ b =n^{-1} \sum_{i=1}^n \{Y^{(i)} - A  X^{(i)} \}$. The value of $\hat A$ can thus be obtained by solving the least-squares problem  \eqref{initial_pb} without intercept and with empirically centered variables.
Given $\hat A$, the matrix $\hat b$ can be recovered using the simple formula $\hat b = n^{-1} \sum_{i=1}^n \{Y^{(i)} - \hat A   X^{(i)} \}$. Consequently, the prediction at point $X$ may be written as
$\hat L  (X) = \overline{Y}^n  + \hat A  (X - \overline{X}^n) $,
where, generically, $ \overline{M}^n = n^{-1} \sum_{i=1}^n  M^{(i)} $.
Similarly, one has $ L^\star (X) = E[Y]  + A^\star (X - E (X) ) $. The previous two expressions emphasize that the excess risk might be decomposed according to $2$ terms: one is dealing with the estimation error on $A^\star$ and one is relative to the error on the averages $E[Y]$ and $E[X]$. We now state this decomposition.

\begin{proposition}\label{first_decomp_ols}
Suppose that Assumption \ref{ass_1} is fulfilled. It holds that
\begin{align*}
\mathcal E (\hat L)  &\leq p^{-1} \{  \| (A^\star -\hat A ) \Sigma^{1/2} \|_F^2 +  2\| E(Y) - \overline Y^n \|_F^2 + 2\| \hat A (E(X) - \overline X^n) \|_F^2\}. 
\end{align*}
where $ \Sigma =  E [ (X-E ( X) )  (X-E ( X) ) ^{T} ]$.
\end{proposition}

The previous proposition will be the key to control the excess risk of both the OLS and LASSO predictors. The OLS (resp. LASSO) estimate, which results from \eqref{initial_pb} when $\lambda = 0$ (resp. $\lambda >0$), is denoted by $(\hat b^{(ols)} ,\, \hat A ^{(ols)} )$ (resp. $(\hat b^{(lasso)} ,\, \hat A ^{(lasso)} )$) and its associated predictor is given by $\hat L^{(ols)}$ (resp. $\hat L^{(lasso)}$).

\subsection{Ordinary least-squares}

In this section, we study the case of the OLS (i.e., when $\lambda = 0$). This will allow to put into perspective the main results of the paper, concerning the LASSO (i.e., when $\lambda > 0$), that will be given in the next section. We now introduce certain assumptions dealing with the distribution of $W^{(i)}$, which depends on $n$ trough $p=p_n$. 
The following one claims that the covariates are uniformly bounded. 
 
\begin{assumption}[bounded variables]\label{ass_4}
With probability $1$, $$\limsup_{n\to \infty} \max _{k\in \mathcal S , \, t\in \mathcal T} | X_{k,t}   | <~\infty. $$
 \end{assumption}


The following invertibility condition on $\Sigma$ can be seen as an identification condition since the matrix $A^\star $ is unique under this hypothesis. 
Denote by $\gamma>0$ the smallest eigenvalue of $\Sigma$. 

\begin{assumption}\label{ass_3}
$\liminf_{n\to \infty}  \gamma >0 $ .
\end{assumption}

 Finally, introduce the noise level $\sigma^2>0$ which consists in a bound on the conditional variance of the residual matrix 
\begin{align*}
\epsilon =\{ Y - L^{\star} (X)\} ,
\end{align*} 
 given the covariates $X$. Formally, $\sigma^2$ is the smallest positive real number such that, with probability $1$,
$ \max _{k\in \mathcal S, \, t\in \mathcal T} E [\epsilon_{k,t}^2 |  X_{\{\cdot,t\}} ]  \leq \sigma^2$. Note that $\sigma^2$ might depend on $p$. The fact that it does not depend on $X$ stresses the homoscedasticity of the regression model. 

\begin{assumption}\label{ass_noise}
$\liminf_{n\to \infty} \sigma >0$.
\end{assumption}

We stress that both previous assumptions on $\gamma$ and $\sigma$ might be alleviated at the price of additional technical assumptions which might be deduced from our proofs.
The following result provides a bound on the excess risk associated to the OLS procedure. The given bound depends explicitly on the quantities of interest $n$ and $p$ as well as on the underlying probabilistic model through  $A^\star$. The asymptotic framework we consider is with respect to $n\to \infty$ and we allow the dimension  $p$ to go to infinity (with a certain restriction). 
A discussion is provided below the proposition.

\begin{proposition}\label{prop:ols}
Suppose that Assumptions \ref{ass_1}, \ref{ass_2}, \ref{ass_4}, \ref{ass_3} and \ref{ass_noise} are fulfilled. Suppose that $n\to \infty$ and $p = p_n \to \infty$ such that $ p\log(p) / n  \to 0$, we have
\begin{align*}
\mathcal E (\hat L^{(ols)}) 
=   O_P \left(  {\frac{ p\sigma^2 +  \| A^\star \|_F^2/p }{n} }  \right) .
\end{align*}

\end{proposition}


The previous bound on the excess risk is badly affected by the parameter $p$. First of all,  the condition $p\log(p) / n \to 0$ implies that $p\ll n$.
Second, the number of nonzero coefficients in $A$ poorly influences the bound.
When $A^\star  = I$, the second term equals $1/n$ and becomes negligible with respect to $p/n$. In contrast, when $A^\star =(1)_{k,\ell}$, i.e., all the covariates are used to predict each output, it equals $p /n$ and becomes a leading term. In between, we have the situation where each line of $A^\star$ possesses only a few non-zero coefficients. Then the magnitude of this term becomes $ 1 / n$.
In this last case, some benefits are obtained when using the LASSO ($\lambda >0$) instead of the OLS ($\lambda = 0$) as detailed in the next section.


\subsection{Regularized least-squares}


The LASSO approach is introduced to overcome the ``large $p$ small $n$'' difficulties of the OLS previously discussed. In contrast with the OLS, the additional $\|\cdot \|_1$-penalty term shall enforce the estimated matrix $\hat A^{(lasso)} $ to be ``sparse'', i.e., to have only a few non-zero coefficients. 
From a theoretical perspective, this will permit 
to take advantage of any sparsity structure in the matrix $A^\star$ associated to $L^\star$.


Introduce the active set $ S^\star_k$ as the set of non-zero coefficients of the $k$-th line of $A^\star$, i.e., for each $k\in \mathcal S $, 
\begin{align*}
 S^\star_k = \left\{\ell \in \mathcal S  \quad : \quad A_{k,\ell}^\star \neq 0\right\}.
\end{align*}
As stated in the following assumption, the sparsity level of each line of $A^\star $ is assumed to be bounded uniformly over $n$. 

\begin{assumption}\label{sparsity_level}
 $\limsup_{n\to \infty } \max_{k\in \mathcal S}  | S_k^\star| <\infty  $.
\end{assumption}

 The following assumption is a relaxation of Assumption \ref{ass_3} which was needed in the study of the OLS approach. For any set $S\subset \mathcal S$, denote by $ S^c$ its complement in $\mathcal S$, and introduce the cone
\begin{align*}
\mathcal C ( S,\alpha ) =\{ u\in \mathbb R^p\, : \, \|u _ {S^c}\|_1 \leq \alpha \|u_{S} \|_1 \}.
\end{align*} 
Denote by $\gamma^\star $ the smallest nonnegative number such that for all $k\in \mathcal S $, we have
$ \|\Sigma^{1/2} u\|_2^2  \geq \gamma^\star \|u\|_2^2$, $\forall u\in \mathcal C (S_k^\star,3)  $.


\begin{assumption}\label{ass:RE}
 $\liminf_{n\to \infty} \gamma^\star >0$.
\end{assumption}

Similar to Assumption \ref{ass_3}, the previous assumption can be alleviated by allowing the value $\gamma^\star $ to go to $0$ at a certain rate which can be deduced from our proof. We are now in position to give an upper bound on the excess risk for the LASSO model.

\begin{proposition}\label{prop:Lasso}
Suppose that Assumptions \ref{ass_1}, \ref{ass_2}, \ref{ass_4}, \ref{ass_noise}, \ref{sparsity_level} and \ref{ass:RE} are fulfilled. Suppose that $n\to \infty$ and $p = p_n \to \infty$ such that $ \log(p) /n \to 0$, we have
\begin{align*}
\mathcal E (\hat L^{(lasso)}) =   O_P \left(  {\frac{   \log(p)    }{n} }  \right) .
\end{align*}
provided that $\lambda = C\sqrt { n\sigma^2 \log(p)}$, for some constant $C>0$.
\end{proposition}

The parameter $p$ which influenced badly the bound obtained on the OLS excess risk is here replaced by $\log (p)$. This shows that without any knowledge on the active variables, the LASSO approach enables to recover the accuracy (at the price of a logarithmic factor) of an ``oracle'' OLS estimator that would use only the active variables. Another notable advantage is that the assumptions for the validity of the bound have been reduced to $\log(p) /n \to 0$ compared to the OLS which needed $p \log(p) / n \to 0$.

The LASSO requires to choose the regularization parameter $\lambda$ which controls the number of selected covariates. In the proof,  we follow the classic approach (as presented in \cite{hastie+t+w:2015}) by choosing $\lambda $ as small as possible but larger than a certain empirical sum involving the residuals of the model; see (11) in the Appendix. This explains our choice $\lambda = C\sqrt { n\sigma^2 \log(p)}$. Note that we could have done differently. Since the Frobenius norm writes as the sum of the $\ell_2$-norms of the matrix lines, problem (4) can be expressed trough $p$ standard LASSO sub-problems each having $T$ outputs. This allows to select different values of $\lambda$ in each sub-problem.  In practice, $\lambda$ will be chosen using cross-validation as explained in the simulation section.


%
%

%

\section{Regime switching}\label{sec:model_switch}

Regime switching occurs when the distribution of road traffic changes after a certain time $t^*\in \{ 1,\ldots,  T \} $. For instance, it might happen that a certain matrix $A^\star$ is suitable to model the morning traffic while another matrix is needed to fit conveniently the afternoon behavior.
%
To account for these potential changes in the distribution, we consider a wider predictive model than the one presented in Section \ref{sec:framework}.
The set of predictors to be considered here writes as  $\cup_{t= 1,\ldots, T} \mathcal L_t $ where each $\mathcal L_t$ is given by
\begin{align*}
\{  b +  (A X_{1},\ldots,A X_{t} ,   A' X_{t+1}  \ldots,   A' X_{T}) \, :\, b  \in \mathbb R^{p\times T},  A \in \mathbb R^{p\times p} , \,  A' \in \mathbb R^{p\times p}  \}.
\end{align*}
Each submodel $\mathcal L_t$ corresponds to a regime switching occurring at time $t$.  Note that the case $t = T$ 
corresponds to the model without regime switching, i.e.,  a single matrix $A$ is used for the whole day.
Within each submodel, the optimal predictor is defined as 
\begin{align*}
L^\star_t = \argmin _{L\in {\mathcal {L}_t}} R(L), 
\end{align*}
and satisfies some normal equations (involving $A$ and $A'$) that can be recovered by applying Proposition \ref{prop:optimal_matrix} with a specific range for $\mathcal T$.
The risk associated to a submodel $t$ is then given by $R( L^\star_t)$. 

If the change point $t^\star$ where given, then the situation would be very similar to what has been studied in the previous section except that two predictors would need to be estimated, one for each time range $U^\star= \{1,\ldots, t^\star\}$ and $V^* = \{t^\star+1,\ldots, T\}$. As each predictor would be computed as
\begin{align*}
&  (\hat b_{t^\star} , \hat A_{t^\star})   \in \argmin_{ b\in  \mathbb R^{p\times t^\star },\, A\in \mathbb R^{p\times p} }   \sum_{i=1}^n  \| Y_{\{\cdot,U^{\star} \}}^{(i)}  - b- A  X_{\{\cdot,U^\star \}}^{(i)} \|_F^2 + 2\lambda \|A\| _1 ,\\
&  (\hat b_{t^\star}', \hat A_{t^\star}')  \in  \argmin_{   b\in  \mathbb R^{p\times (T- t^\star )},\, A\in \mathbb R^{p\times p} }   \sum_{i = 1}^n  \| Y_{\{\cdot, V^{\star} \}}^{(i)}   -b-   A  X_{\{\cdot, {V^{\star }} \}}^{(i)} \|_F^2 + 2\lambda \|A\| _1 ,
\end{align*} 
the results obtained in the previous section would apply to both estimates separately and one would easily derive an excess risk bound when the change point $t^\star$ is known. The key point here is thus to recover the change point $t^*$ by some procedure. Next, we study a model selection approach based on cross-validation to estimate $t^\star$. 


Suppose that $T\geq 2$. For each $t\in \{1,\ldots, T\}$, define $U_t= \{1,\ldots, t\}$, ${V_t} = \{t+1,\ldots, T\}$ and let $\mathcal F= \{I_1,\ldots, I_K\}$ be a partition of $\{1,\ldots, n \}$ whose elements $I_k$ are called folds. We suppose further that each fold contains the same number of observation $n/K$. For each fold $I\in \mathcal F$, define
\begin{align*}
&  (\hat b _{J,t}, \hat A _{J,t} )    \in \argmin_{  b\in  \mathbb R^{p\times t },\,  A\in \mathbb R^{p\times p} }   \sum_{i\in J}  \| Y_{\{\cdot,U_t\}}^{(i)}  - b-  A  X_{\{\cdot,U_t\}}^{(i)} \|_F^2 + 2\lambda \|A\| _1 ,\\
&  (\hat b _{J,t}', \hat A _{J,t}' )   \in  \argmin_{ b\in  \mathbb R^{p\times (T- t)},\,   A\in \mathbb R^{p\times p} }   \sum_{i\in J}  \| Y_{\{\cdot, {V_t} \}}^{(i)}  -b  -  A  X_{\{\cdot, {V_t} \}}^{(i)} \|_F^2 + 2\lambda \|A\| _1 ,
\end{align*} 
where $J$ is the complement of $ I $ in $\{1,\ldots, n\}$. To lighten the notations, the superscript $(Lasso)$ will be from now on avoided. The estimate of the risk based on the fold $I$ is defined as
\begin{align*}
&\hat R _{I,t} = \frac{K}{ np  }   \sum_{i\in  I }  \{ \| Y_{\{\cdot,U_t\}}^{(i)}   -  \hat b_{J,t} -  \hat A_{J,t}   X_{\{\cdot,U_t\}}^{(i)} \|_F^2 + \| Y_{\{\cdot,V_t\}}^{(i)}   -  \hat b_{J,t}' - \hat A_{J,t}'   X_{\{\cdot,V_t\}}^{(i)} \|_F^2 \} .
\end{align*}
  The resulting cross-validation estimate of  $R_t$ is then the average over the folds $I\in \mathcal F$ of the risks $\hat R _{I,t}$ and the estimated  change point $\hat t$ is the one having the smallest estimated risk, i.e.,
\begin{align*}
&\hat t \in \argmin _ {t = 1,\ldots,  T }  \hat R _{t} :=  \frac{1}{ K }   \sum_{I \in \mathcal  F }  \hat R _{I,t}    .
\end{align*}
This $\hat t$ represents the best instant from which a different (linear) model shall be used.
The following proposition provides a rate of convergence for the cross-validation risk estimate.

\begin{proposition}\label{prop:regime_switch1}
Under the assumptions of Proposition \ref{prop:Lasso}, we have for all $t\in \{1,\ldots T-1\}$,
\begin{align*}
  |\hat R _{t} - R( L^\star_t) |  = O_P \left(  \sqrt {  \frac{ \log(p) }  {n }  }   \right).
\end{align*}
\end{proposition}

The difficulty in proving the previous result is to suitably control for the large dimension in the decomposition of the error. This is done by relying on a bound on the $\ell_1$-error associated to the matrix estimate $\hat A_n ^{(lasso)}$; see Proposition \ref{prop:l1_bound} in the Appendix. 

Applying the previous proposition allows to establish the consistency of the cross-validation detection procedure.

\begin{corollary}\label{cor:regime_switch2}
Under the assumptions of Proposition \ref{prop:Lasso}, suppose there exists $t^*\in \{1,\ldots T-1\}$ such that $\limsup_{n\to \infty} R( L^\star_{t^\star} )  < \liminf_{n\to \infty} R( L^\star_t) $ for all $t\neq t^*$, then  $$P (\hat t = t^* ) \to 1 .$$
\end{corollary}


Interestingly, the previous regime switching detection can be iterated in a greedy procedure in which, at each step,  we decide (based on cross-validation) if a regime switching is beneficial. If it does, we continue. If it does not, we stop. The study of such an iterative algorithm represents an interesting avenue for further research. 



%

\section{Real data analysis}\label{sec:simu}
This section provides a comparative study of different methods. First, real world data is presented followed by the different methods in competition and an analysis of the obtained results. Finally, graphical interpretations of the proposed LASSO approach are provided. A GitHub repository is available\footnote{\url{https://github.com/mohammedLamine/RS-Lasso}}. 


\subsection{Dataset}
\label{Dataset}

The initial dataset contains the speed (km/h) and the location of each car using the Coyote navigation system (Floating car data). Data was collected in the Rennes road network from December 1, 2018 until July 9, $2019$ every $30$ seconds. Signals were received from $113577$ vehicle.

Some pre-processing is carried out to correct sensor errors and structure the data in a convenient format. First, the received locations are GPS coordinates whereas the model introduced here is based on the road section. Accordingly, we map-match the locations to the OpenStreetMap road network dataset \citep{greenfeld2002matching} in order to obtain the section of each car location data. Second, to obtain data representing the flow in the sections of the network, we aggregate and average the observed speed values of cars over $15$-minute intervals for each section. Third, for some sections and time intervals, no value is observed. We therefore impute these values using the historical average associated to the time intervals and sections. To lower the number of missing values, we focus on a particularly busy time period from $3$pm to $8$pm (local time) of each workday. It is worth noting that this dataset comes from a low-frequency collection, meaning that for a given time and section, the average number of observations on the selected sections of the network is about $6$ logs per $15$-minutes per section. This is due to the fact that we only observe a portion of the vehicles in the network (those equipped with Coyote GPS devices), and that suburban sections have low traffic. Thus, aiming for even higher precision (less than $15$ minutes) would increase the number of missing values. 

In summary, we consider only the $15$-minute time intervals in the period from $3$pm to $8$pm (local time) of each workday only. The resulting dataset corresponds to our $W^{(i)}$, $i=1,\ldots, n$, matrices where $p=556$, $n=144$ and $T + 1=20$ (Figure \ref{fig:dataset_snap_shot}). 
We divide the data into $3$ subsets: train $63$\% ($1710$ examples), validation $27$\% ($741$ examples), test $10$\% ($285$ examples), splits are made such that days are not cut in the middle and all subsets contains sequential full days data. 

\subsection{Comparative study}

\subsubsection{Methods}

The methods used in this study can be divided into two groups reflecting the information that is used to predict future values. 
In the first group, baselines are used to predict the speed using only data from the same section. 
The second group uses data collected from the entire network. This includes the linear predictors that have been previously studied and also certain neural network predictors that we introduce for the sake of completeness.

\subsubsection*{Baselines}
\begin{itemize}
  \item \textit{Historical average} (HA): for each section $k$ and time interval $t$, the prediction is given by the average speed at time interval $t$ and section $k$. The averaged speed is computed on the training dataset.
  \item \textit{Previous observation} (PO): for each section $k$ and time interval $t$, the prediction is given by the observed speed at $t-1$ and section $k$.
  \item \textit{Autoregressive} (AR): 
 for each section $k$ and time interval $t$, the prediction is given by a linear combination of the $t_0$ previous observed time speed $t-1,\ldots, t-t_0$ in section $k$. The coefficients are estimated using \textit{ordinary least-squares} on the training dataset. In our approach, we use the implementation from the "Statsmodels" package and vary the order $t_0$ from $1$ to $5$.
\end{itemize}

\subsubsection*{Linear predictors}
\begin{itemize}
  \item \textit{Ordinary least-squares} (OLS): we first compute the linear predictor for each section separately. This corresponds to independently solving $p$ sub-problems using OLS. Then we stack the solutions $(\beta_k)_{1\leq k\leq p} \subset \mathbb R^{p}$ to form the solution $A^T = (\beta_1,\ldots \beta_p)  \in \mathbb R^{p\times p}$ of the main problem.
  \item \textit{LASSO}: We compute the LASSO predictor in the same way as the OLS by independently solving $p$ sub-problems. For each sub-problem, we select the regularization coefficient $\lambda$ using the $5$-fold cross-validation from the \texttt{sklearn} package on our training data. Taking the same value for $\lambda$ in all sub-problems (as is case in the theoretical analysis) does not have an impact on the results. 
  \item \textit{Time-specific LASSO} (TS-LASSO): We subdivide the problem into smaller time-specific problems that we solve separately. The algorithm finds a different coefficients matrix $A$ for each time $t \in \{1,\ldots, T\}$. We compute each of these predictors in the same way as the LASSO.  
  \item \textit{Regime switching LASSO} (RS-LASSO): 
This approach is introduced in Section \ref{sec:model_switch} and consists of searching for the best time instant $\hat t$ from which we should use another matrix $A$ for predictions. For each $t \in\{1,\ldots, T\}$, we run a first LASSO procedure over the time frame $\{1, \dots, t\}$ and another over $\{t+1, \dots, T\}$. For each $t$, the resulting risk is estimated using cross-validation and finally, we find $\hat t$ as the time change having the smallest estimated risk. 
  \item \textit{Group LASSO} (Grp-LASSO): The Grp-LASSO is similar to the LASSO but uses a different penalization function which involves the norm of each column of the coefficients matrix $A$. This approach filters out sections that are found to be irrelevant to predict the complete network and hence is useful when the same sparsity structure is shared among the different sections.
\end{itemize}

\subsubsection*{Neural networks predictors}
  \begin{itemize}
  \item \textit{Multilayer perceptron} (MLP) made of $2$ fully connected hidden layers. Data is normalized at the entry point of each hidden layer.  All model parameters are regularized using an $\ell_1$ penalty. 

  \item Fully connected \textit{long short-term memory} (FC-LSTM) composed of one LSTM layer followed by a fully connected layer (output layer) with $2$-time steps. The model parameters are regularized using an $\ell_1$ penalty. (note that we do not apply it to the recurrent step in the LSTM). 
\end{itemize}


These neural networks are built using a hyperbolic tangent as the activation function except for the output layer for which the identity function is used. The input data values are scaled so that they belong to $[-1,1]$. We train our models using an ADAM optimizer and \textit{mean squared error loss}. Because of the high dimensionality of the model ($p = 556$), all non-regularized approaches overfit the training dataset. To asses this problem, we rely on a combination of $2$ regularization techniques:
\begin{itemize}
\item Earlystopping \cite{caruana2001overfitting} is used to stop model training when the model shows signs of overfitting. 
\item $\ell_1$-regularization is conducted by adding a penalization term to the loss function that enforces a sparse structure of the parameters.
\end{itemize}
We stress that other configurations have been tested without improving the results: (a) using the sigmoid,  the ReLU and its variants as activation functions; (b) Log-scaling the input data and scaling to the interval $[0,1]$ rather than $[-1,1]$, (c) varying the number of layers, neurons, and lags, and (d) other regularization techniques. The models are built using the package Keras with TensorFlow backend in python.

\subsubsection{Results}

The methods are compared using the mean-squared error (MSE) and the mean absolute error (MAE). Both are computed on the test set of the data. We examine these errors, covering all of the sections as well as pin-pointing specific segments; we do the same over the entire time period, and focus on more specific periods, such as particular days, as well. 
\subsubsection*{Performance on the whole road network}
\begin{table}[h!]
  \centering
  \resizebox{\textwidth}{!}{%
  \begin{tabular}{lrrrrrrrrrrrll} 
    \hline
     \textbf{Model}  & \textbf{AR 1}  & \textbf{AR 3}  & \textbf{AR 5}  & \textbf{FC-2LSTM}  & \textbf{HA}  & \textbf{LASSO}   & \textbf{MLP}  & \textbf{OLS}  & \textbf{PO}  & \textbf{RS-LASSO}  & \textbf{TS-LASSO}  & \textbf{Grp-LASSO}   \\ 
    \hline
     \textbf{MAE}    & 7.25           & 7.23           & 7.23           & 7.21               & 7.72         & \textit{7.14}    & 7.26          & 9.20          & 9.38         & \textbf{7.13}      & 7.34               & 7.30                            \\ 
    \hline
     \textbf{MSE}    & 113.85         & 113.22         & 113.22         & 111.44             & 133.41       & \textit{109.24}  & 113.57        & 162.79        & 183.38       & \textbf{109.04}    & 115.47             & 113.81                      \\
    \hline
    \end{tabular}%
  }

\caption{MSE and MAE computed over the entire network and the entire time period.}
\label{tab:experiments_results}
\end{table}

\begin{figure}[h!]
  
  \begin{subfigure}[b]{0.45\textwidth}
    \includegraphics[width=1\textwidth]{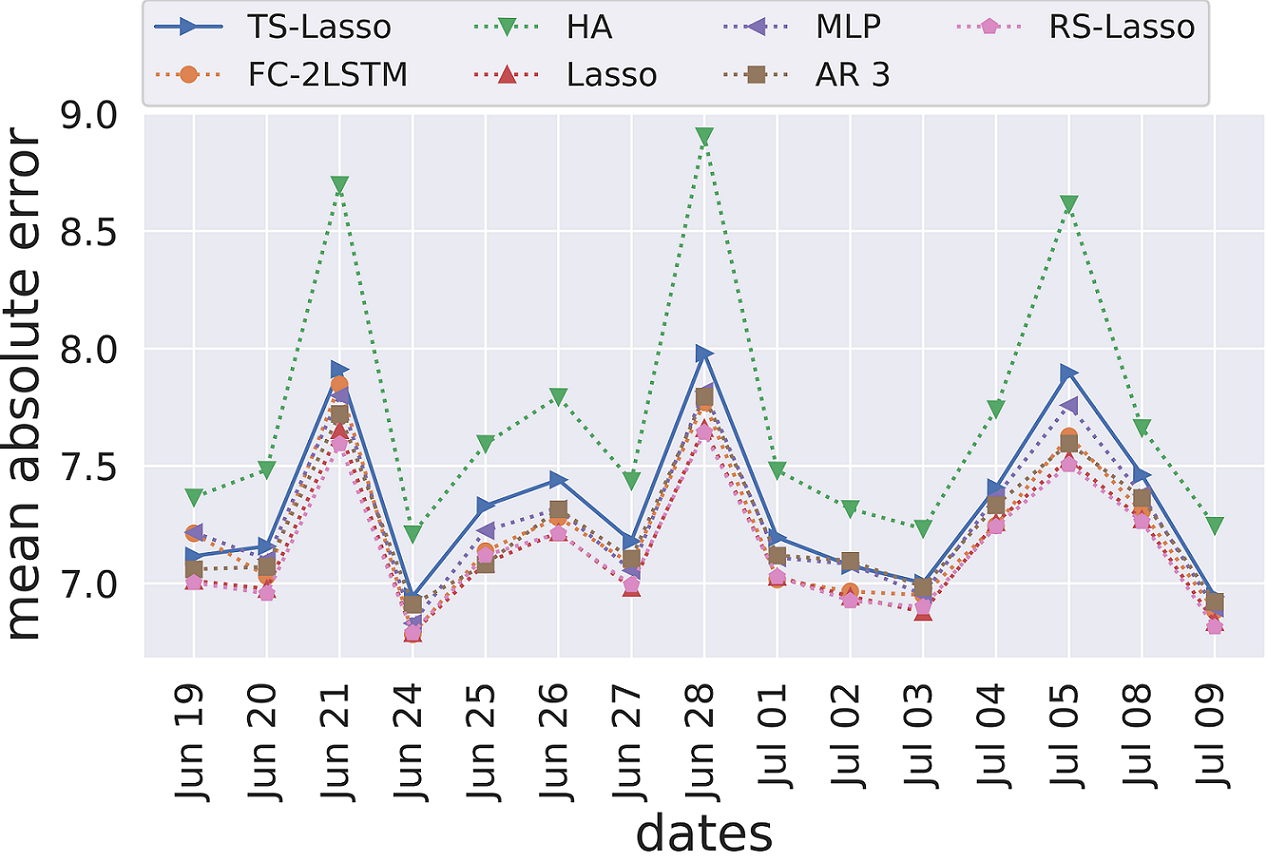}
    \caption{}
  \end{subfigure}
  \hfill
  \begin{subfigure}[b]{0.5\textwidth}
    \includegraphics[width=1\textwidth,keepaspectratio]{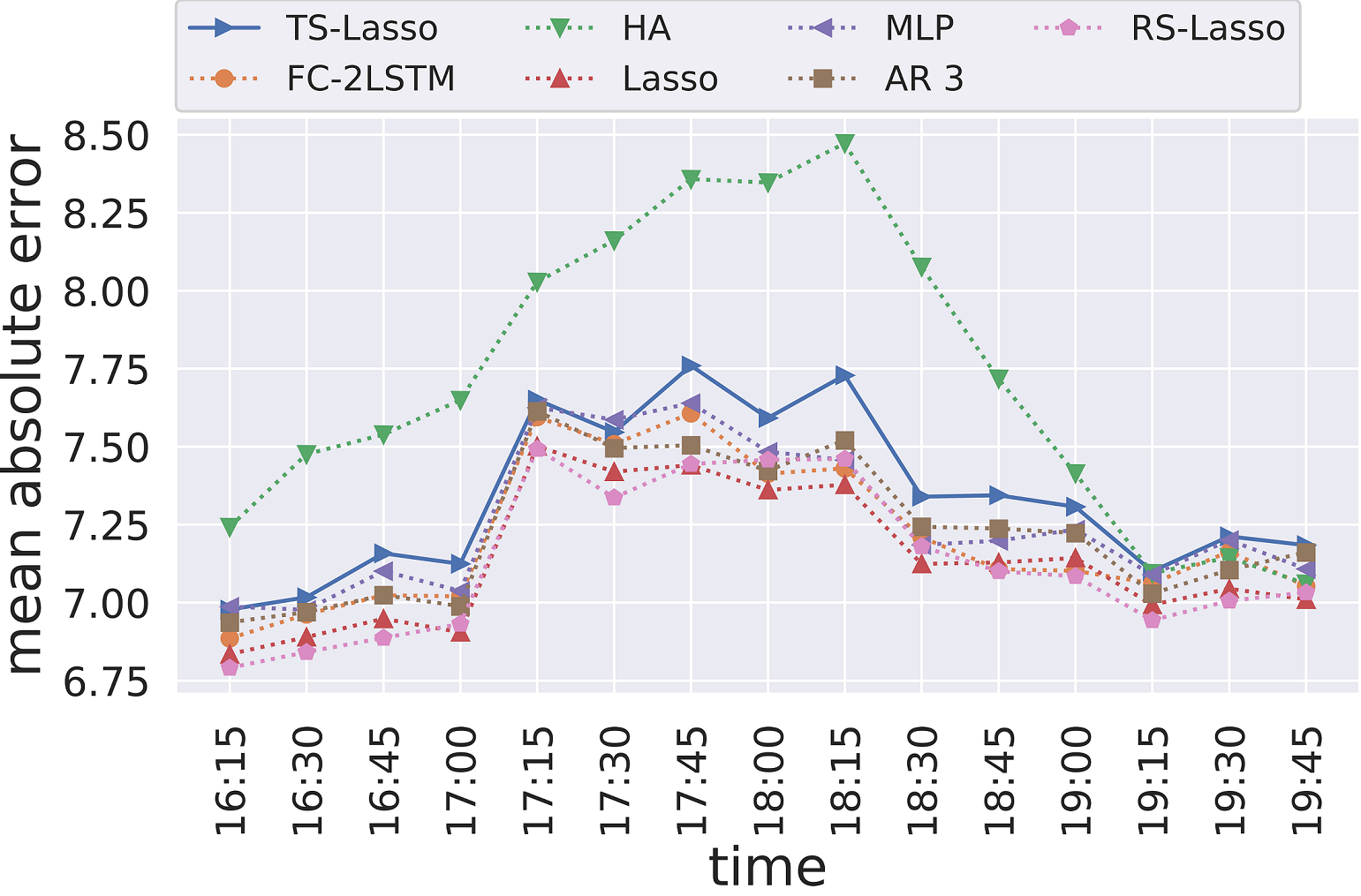}
    \caption{} 

  \end{subfigure}

\caption{MSE and MAE computed (a) for each day and (b) for each specific hour.}
\label{fig:experiments_results}
\end{figure}

Table \ref{tab:experiments_results} shows the performance of all models over the entire network. The simple OLS model does not generalize well together with PO prediction; they yield the worst performances. The OLS is subjected to an over-fitting problem as the MAE is equal to $9.20$ (resp. $5.69$) on the test set (resp. train set). Other baseline performances are similar with AR, yielding a slightly better performance than the HA. Note that increasing the number of lags does not lead to better results. The neural networks offer close results to those of the AR with the LSTM performing slightly better than the MLP model. The different LASSO approaches also yielded variable results; the model with regime-switching (RS-LASSO with a time switch at $6$:$15$pm) performs better than all other models ($7.13$ MAE) and is just slightly better than the LASSO ($-0.01$ MAE). Using a different LASSO predictor for each time (TS-LASSO) does not improve the error and yields a worse performance than the LASSO ($+0.20$ MAE). The same observation on Grp-LASSO suggests that entirely filtering some sections causes a certain loss. 

In Figure \ref{fig:experiments_results}, we present (a) the error computed for each specific day  (large errors correspond to fridays) and (b) the error for each specific hour (averaged over the days). For readability, we only present a subset of the models. The figure shows the LASSO and the RS-LASSO results over days and time and confirms that they outperform all other models. The two approaches alternate performance over time, we notice that the RS-LASSO performs better at the first and the last $1.5$ hours and is worse than the LASSO in the middle hours. As this cut-off is happening around the switch to the RS-LASSO ($6$:$15$pm), it might be possible that the RS-LASSO parts (before and after switching) focus on modeling the regime of the corresponding periods and overlooks the middle period that is separated by the time switch. 

\subsubsection*{Performance on highly variable sections }
To further understand the results, we compare the models on highly variable sections. Variable sections are selected based on a clustering algorithm applied to descriptive statistics (Min, Max, Mean, standard deviation, Quartiles) and computed for each section at each time. The cluster represents about $30\%$ of the sections.
\begin{table}[h!]
  \centering
  \resizebox{\textwidth}{!}{%
  \begin{tabular}{lrrrrrrrrrrrl} 
    \hline
     \textbf{Model}  & \textbf{AR 1}  & \textbf{AR 3}  & \textbf{AR 5}  & \textbf{FC-2LSTM}  & \textbf{HA}  & \textbf{LASSO}   & \textbf{MLP}  & \textbf{OLS}  & \textbf{PO}  & \textbf{RS-LASSO}  & \textbf{TS-LASSO}  & \textbf{Grp-LASSO}   \\ 
    \hline
     \textbf{MAE}    & 9.27           & 9.26           & 9.27           & 9.11               & 10.18        & \textit{9.01}    & 9.17          & 11.50         & 11.81        & \textbf{8.94}      & 9.20               & 9.17                            \\ 
    \hline
     \textbf{MSE}    & 172.59         & 171.94         & 171.98         & 163.45             & 207.18       & \textit{161.29}  & 165.39        & 238.91        & 274.92       & \textbf{159.66}    & 168.01             & 165.28                        \\
    \hline
    \end{tabular}
  }

\caption{MSE and MAE restricted to highly variable sections and computed (a) for each day and (b) for each specific hour.}
\label{tab:experiments_results_variable_roads}
\end{table}
\begin{figure}[h!]

\begin{subfigure}[b]{0.45\textwidth}
  \includegraphics[width=1\textwidth]{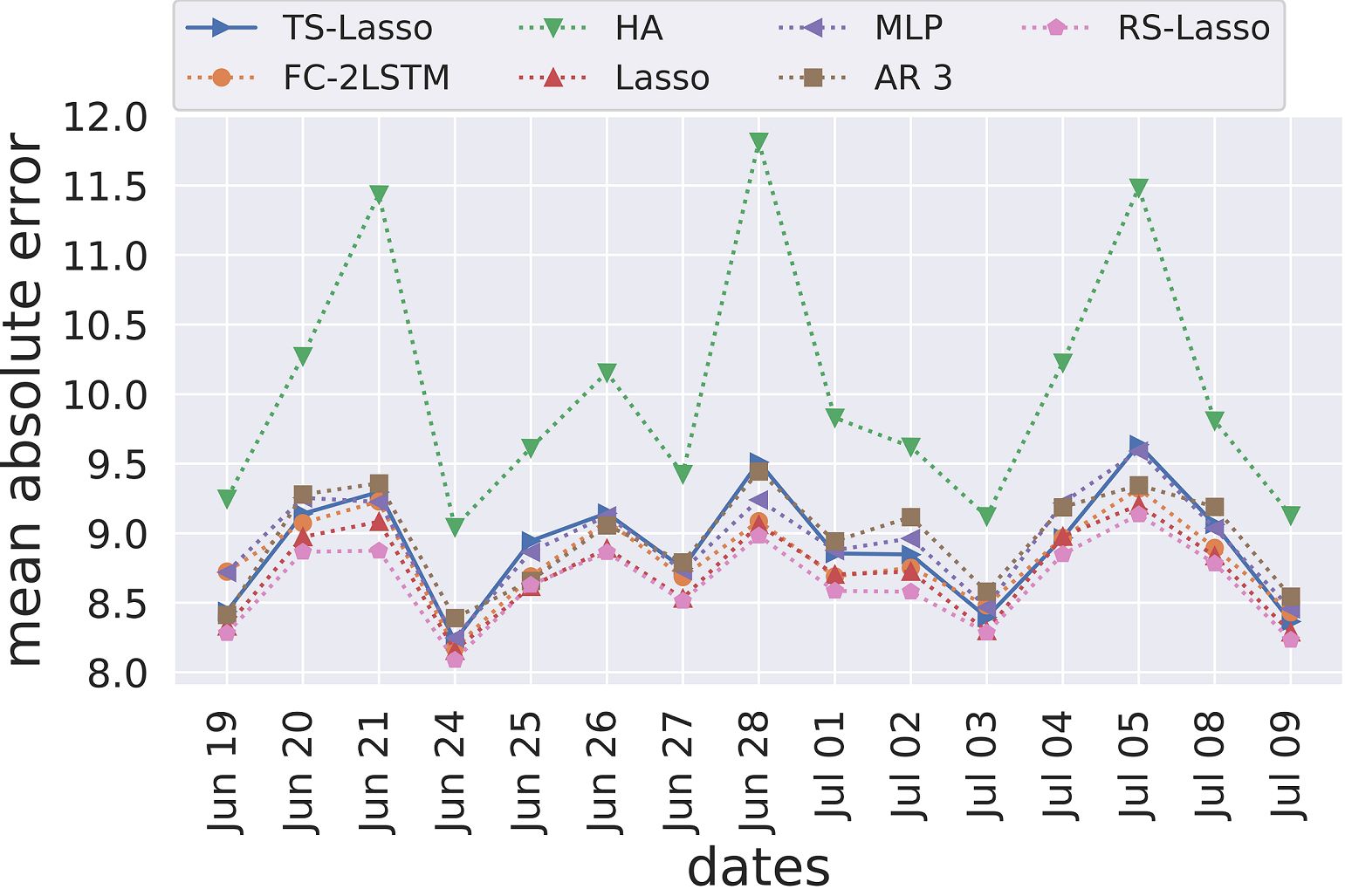}
  \caption{}
\end{subfigure}
\hfill
\begin{subfigure}[b]{0.5\textwidth}
  \includegraphics[width=1\textwidth,keepaspectratio]{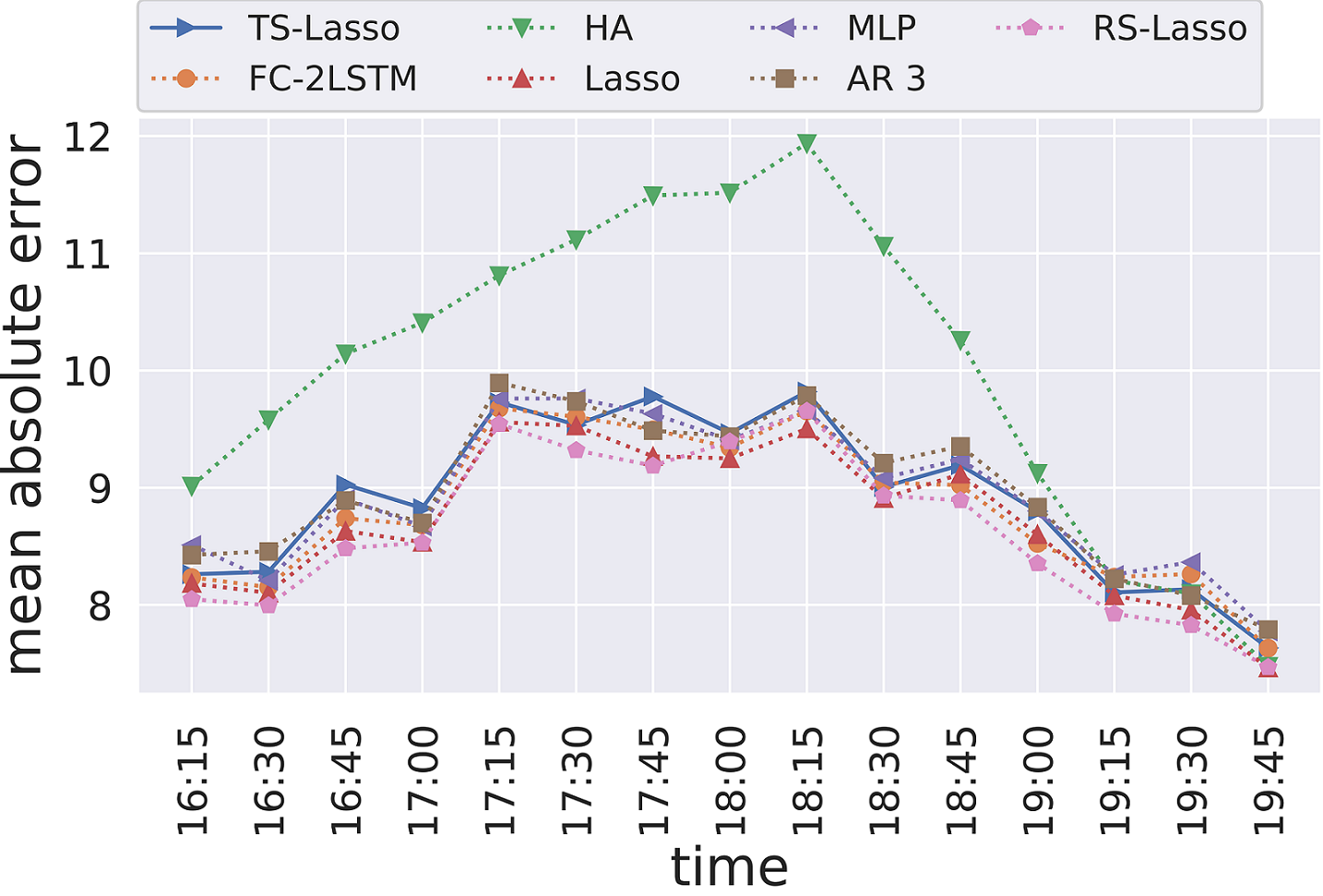}
  \caption{} 

\end{subfigure}

\caption{The prediction error for each method (a) over days, (b) over time.}
\label{fig:experiments_results_variable_roads}
\end{figure}

Table \ref{tab:experiments_results_variable_roads} and Figure \ref{fig:experiments_results_variable_roads} show the performance over highly variable sections only. All model performances deteriorate compared to the performances on the full road network. Similar observations as before can be made on these results with the exceptions that both neural networks, TS-LASSO and Grp-LASSO, perform better than the autoregressive models ($ -0.2$ MAE) . The performance gap between the models is wider. The RS-LASSO is clearly better than the LASSO ($-0.07$ MAE) compared to the observed performance on the entire network. The alternation of performance over time is still observable between the LASSO and RS-LASSO models.

\subsubsection*{Comparison of different penalties}

We now compare the performances between the LASSO, Ridge and Elastic Net penalties defined as:
\begin{itemize} 
  \item LASSO ($\ell_1$) penalty :  $w\mapsto\lambda  ||w||_1$.
  \item Ridge  ($\ell_2$) penalty : $w\mapsto \lambda  ||w||^2_2$.
  \item Elastic Net penalty:  $w\mapsto \lambda  (\alpha  ||w||_1 +  (1 - \alpha) ||w||^2_2)$.
\end{itemize} 
where $\lambda$ is the regularization parameter and $\alpha \in [0.1,0.9]$ represents some ratio between the LASSO and Ridge penalties.  For each method, $\lambda $ and $\alpha$ are determined through the application of a $5$-fold cross-validation. For Elastic Net, the best $\alpha$ was around $0.6$ giving the $\ell_1$-penalty at a slightly higher weight. 
\begin{table}[h!]
  \centering
  \begin{tabular}{lllll}
    \hline
    \textbf{Penalty} & \textbf{OLS (No penalty)} & \textbf{LASSO} & \textbf{Ridge} & \textbf{ElasticNet} \\ \hline
    \textbf{MSE}     & 162.83                    & 105.89         & 110.78         & 108.48              \\ \hline
    \textbf{MAE}     & 9.22                      & 7.00           & 7.15           & 7.09                \\ \hline
    \end{tabular}
  \caption{MSE and MAE computed over the whole network and the whole time period for the penalties LASSO, Ridge, Elastic Net.}
  \label{reg_res_compar}
\end{table}

Table \ref{reg_res_compar} shows the obtained MAE and MSE for all the models. The LASSO penalty shows a slightly better performance than the other two penalties. This can be explained by the ability of the LASSO to effectively select features for the model which is increasingly important in a sparse, high dimensional setup. Therefore, choosing the $\ell_1$-penalty for this problem seems to be efficient and also gives a better interpretation of the model's parameters.

\subsection{Graphical visualization of the LASSO model}

In this section, the aim is to show that the LASSO coefficients are interpretable and can be used to capture certain features of the architecture of the road network.

The LASSO coefficients are indicators of the traffic flow dependencies between the sections of the road network. For a section of interest, say $k$, a way to visualize these connections is to draw, on top of the network map, arcs going from section $k$ to sections for which the corresponding regression coefficient is non-zero (that belongs to the active sets $S^\star_k$). In addition, the color intensity and the width of each arc can be scaled by the value of the corresponding coefficient in order to highlight the influence of one section on another.

\begin{figure}[h!]
  
  \begin{subfigure}[b]{0.49\textwidth}
    \includegraphics[width=1\textwidth]{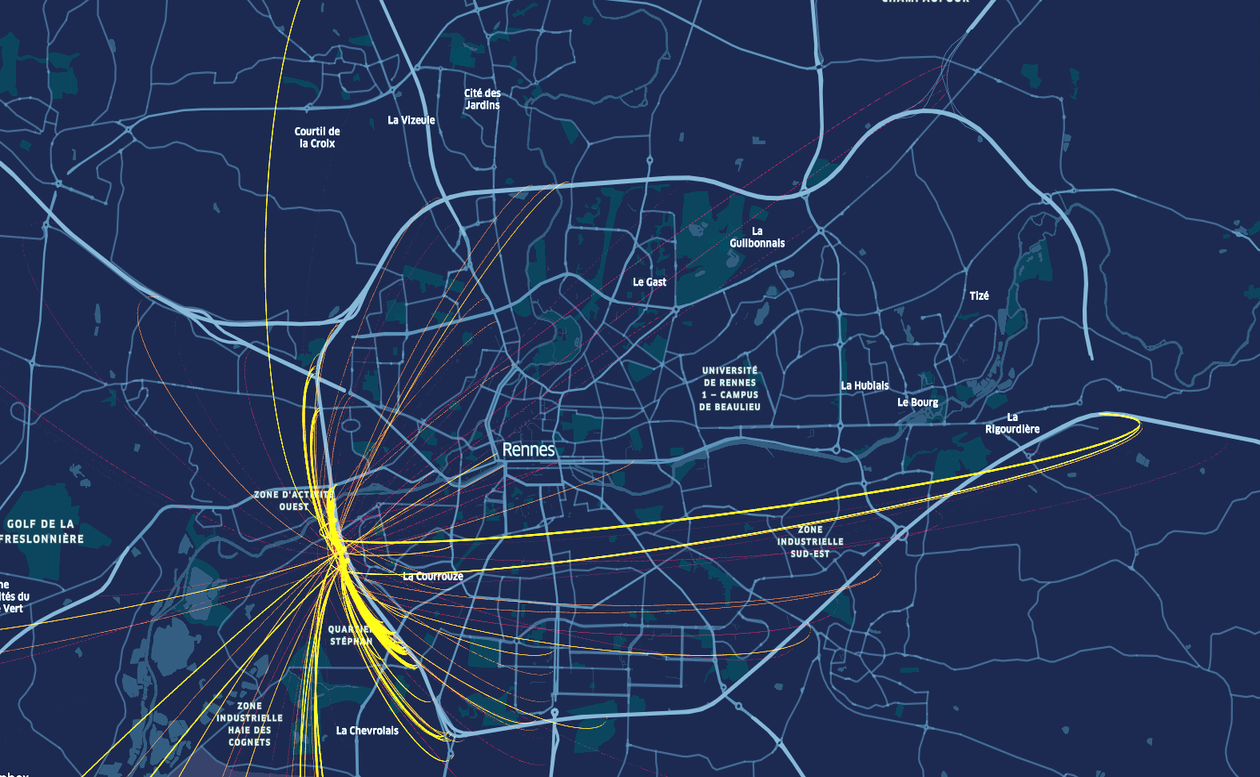}
    \caption{}
  \end{subfigure}
  \hfill
  \begin{subfigure}[b]{0.49\textwidth}
    \includegraphics[width=1\textwidth]{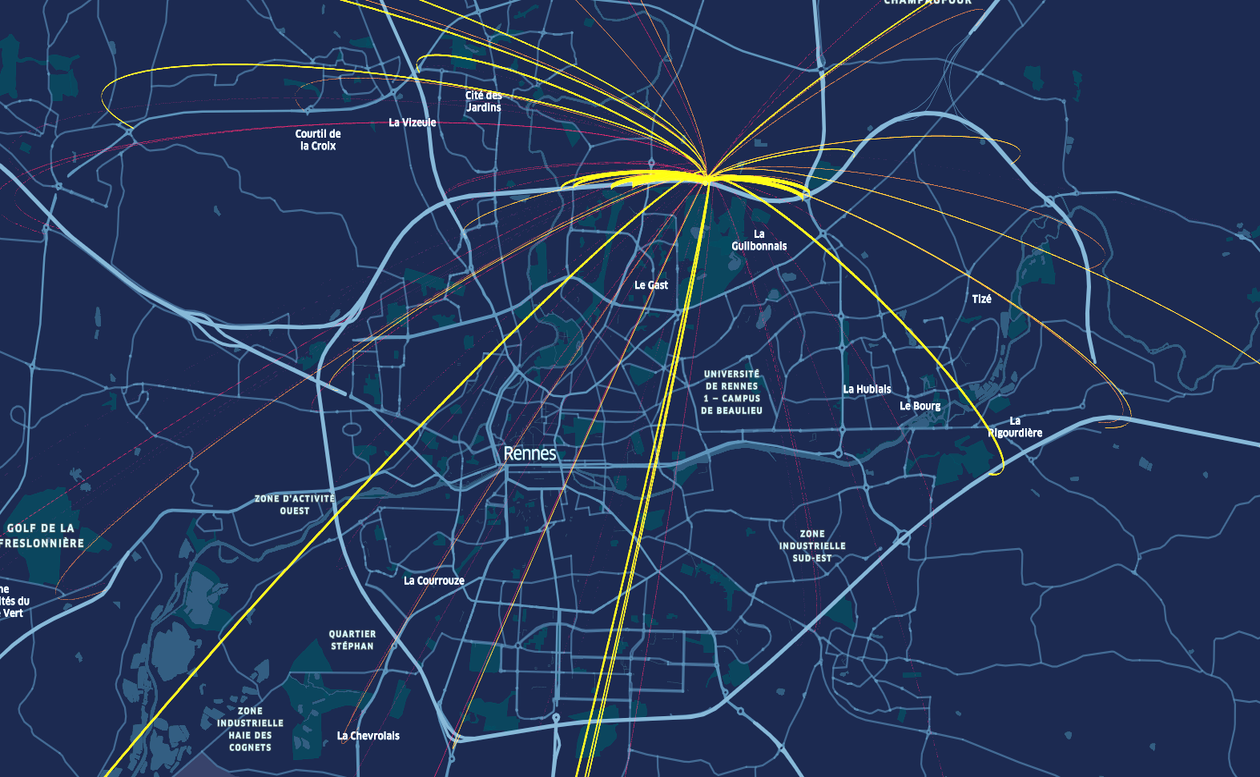}
    \caption{}
  \end{subfigure}

  \begin{subfigure}[b]{0.49\textwidth}
    \includegraphics[width=1\textwidth]{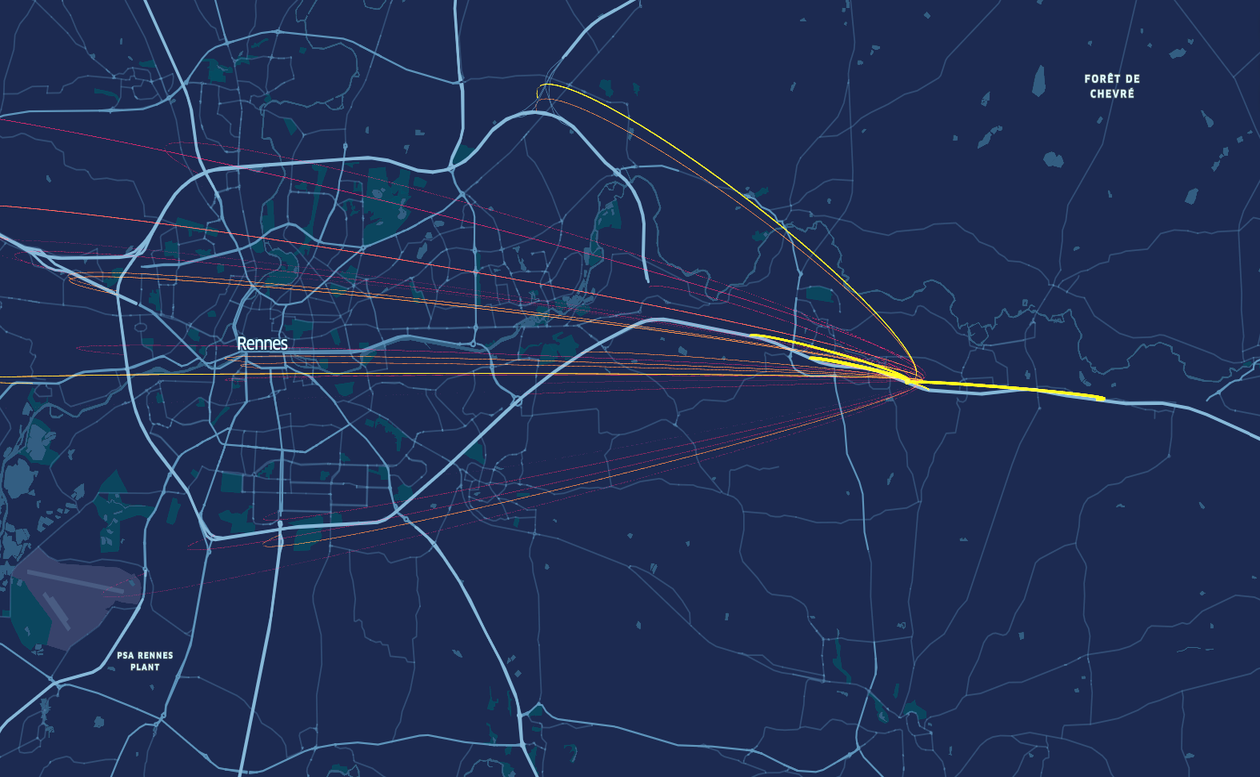}
    \caption{}
  \end{subfigure}
  \hfill
  \begin{subfigure}[b]{0.49\textwidth}
    \includegraphics[width=1\textwidth]{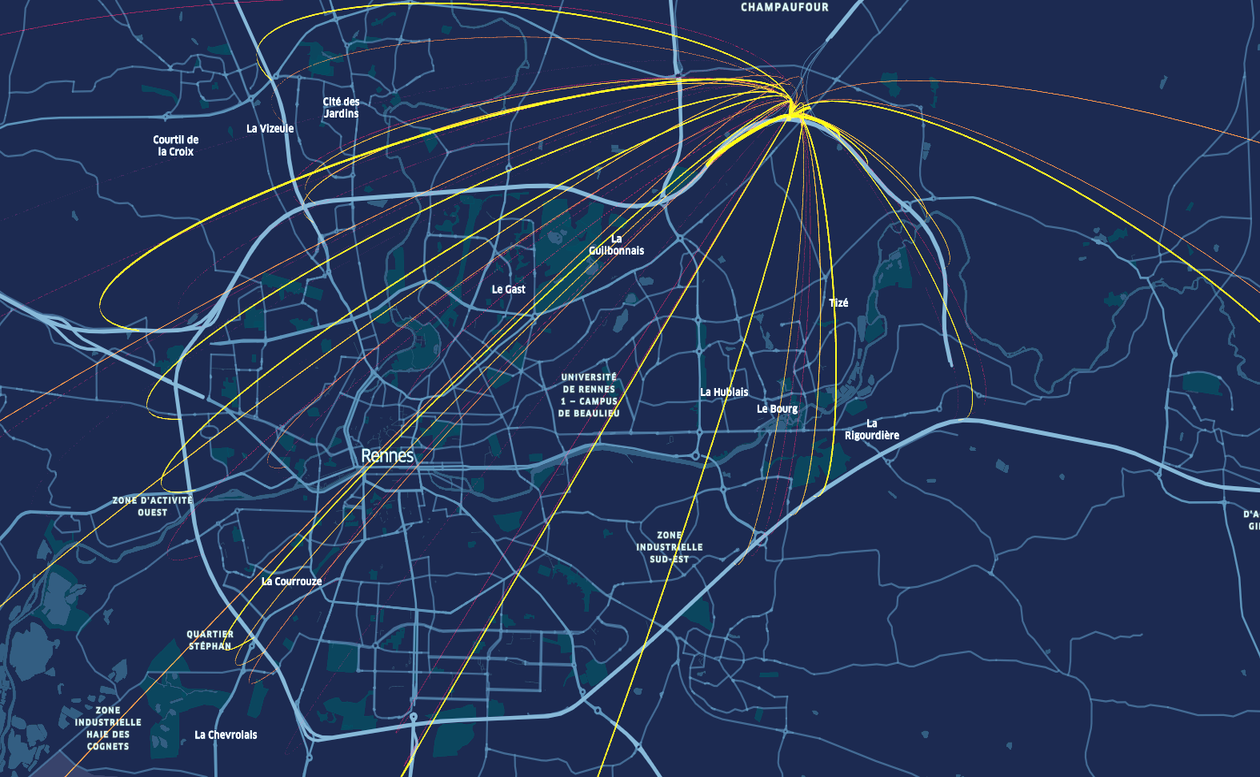}
    \caption{}
  \end{subfigure}
  \begin{subfigure}[b]{0.49\textwidth}
    \includegraphics[width=1\textwidth]{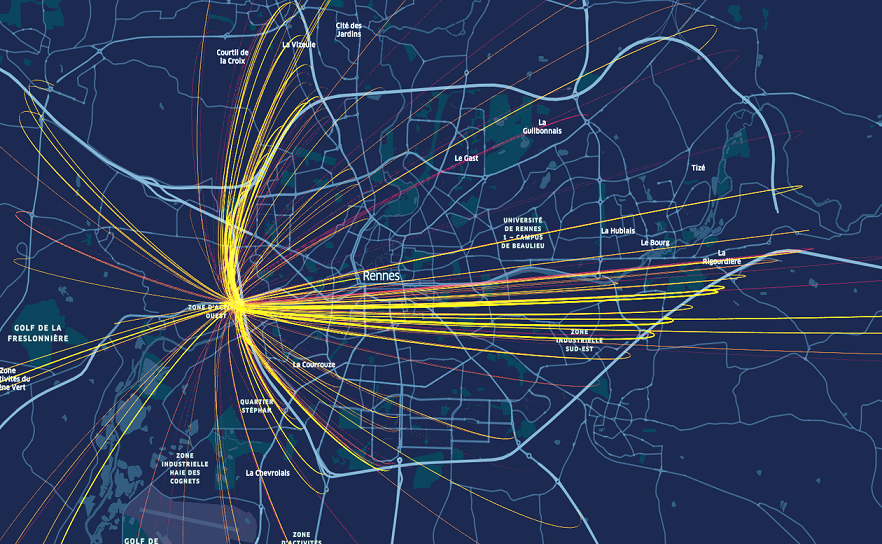}
    \caption{}
  \end{subfigure}
  \hfill
  \begin{subfigure}[b]{0.49\textwidth}
    \includegraphics[width=1\textwidth]{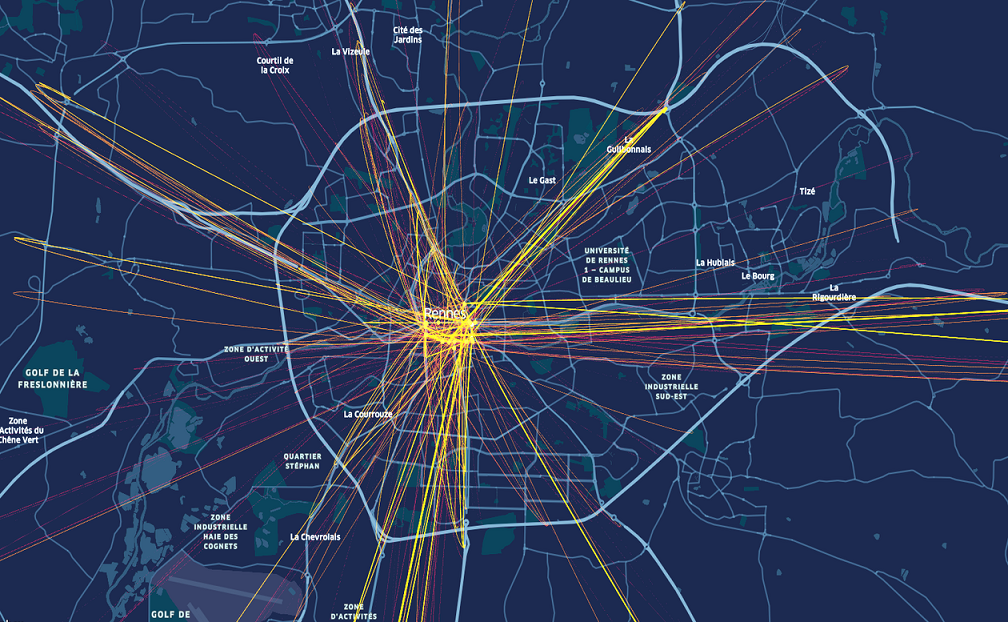}
    \caption{}
  \end{subfigure}
  \caption{
For each graph, arcs are drawn from the section of interest to the active sections deduced from the estimated matrix coefficients $\hat A^{(lasso)}$. Color intensity and width are scaled according to the value of the matrix coefficients. Different types of section are considered: (a) and (b) ring roads; (c) highway; (d,e) link roads; (f) city center.}
  \label{fig:Lasso_visualization}
  \end{figure}

The previous visualization is provided for several representative sections in Figure \ref{fig:Lasso_visualization}. In Figure \ref{fig:Lasso_visualization}, (a, b) represents sections from the Rennes ring roads, or beltways. These roads have an upper limit of $70km/h$ and circles the whole city of Rennes. We see that these sections strongly connect to adjacent sections in the ring and also to the closest highways. This is expected as these sections are physically connected to each other as they form a ring road and have links to adjacent highways. 
We also observe this kind of behavior on the highways (Figure \ref{fig:Lasso_visualization}.c) where the sections of the highway have a higher dependence on adjacent sections of the highway.
In Figure \ref{fig:Lasso_visualization}, (d, e), we show sections that are part of link roads, or bypasses. We see that these sections have a high number of connections to adjacent ring roads. These link roads usually connect the flow between relatively distant parts of the city so we also observe farther connections around the ring network. We observe similar aspects with sections from the city center as shown in Figure \ref{fig:Lasso_visualization}, (f). 

Another meaningful visualization provided by the model consists of representing the influence of each section on the overall network using the values of the estimated coefficients. 
For section $\ell$, the criterion is given by  $C_\ell = \sum_{1\leq k\le p, \hat A_{k\ell} ^{(lasso)} >0}  A_{k\ell}$. Taking into account negative coefficients implies an inverse relationship between sections, meaning that when a section is getting congested, another is getting less traffic. This may happen when there are incidents on the road but does not characterize the average influence between sections. Thus we only consider positive coefficients in building the criterion (consisting of $93\%$ of the coefficients).

\begin{figure}[h!]
  \hspace*{3cm}
  \begin{subfigure}[b]{0.5\textwidth}
    \includegraphics[width=1\textwidth]{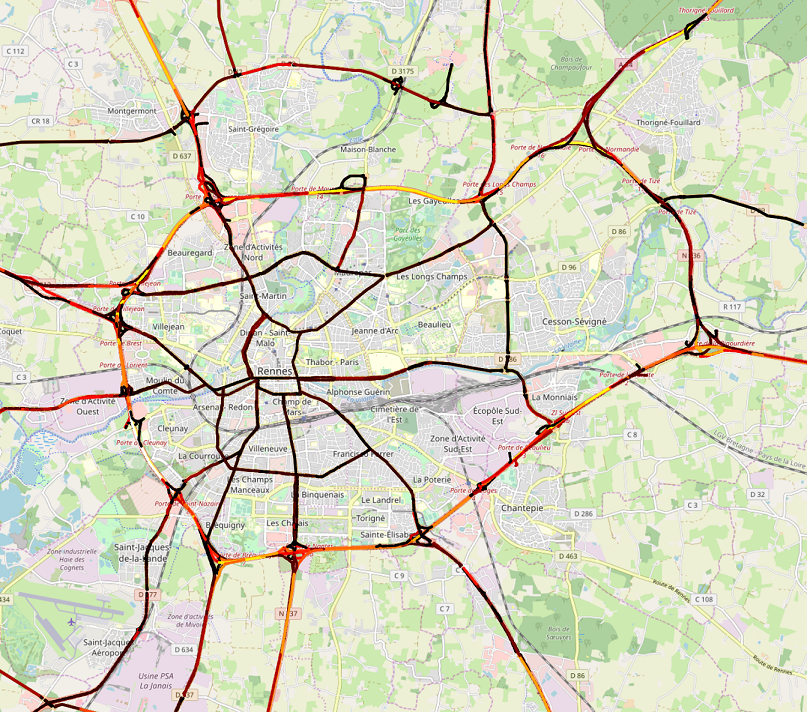}
  \end{subfigure}
  \hspace*{0.5cm}
  \begin{subfigure}[b]{0.11\textwidth}
    \centering
    \includegraphics[width=1\textwidth]{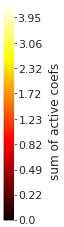}
  \end{subfigure}
  
\caption{Heatmap plot of the sections based on the influence criterion $C_\ell$.}
\label{fig:sections_influence_std}

\end{figure} 

  


Figure \ref{fig:sections_influence_std} is a heatmap of the road network where each section is color-coded according to the influence coefficient $C_\ell$.
We see that the network behavior is mainly influenced by ring road sections and highways. This is expected as traffic mainly flows on these roads as opposed to most sections in the city center, and those in the rural and suburban areas, that carries smaller quantities of traffic. 

  






\section{Simulation study}\label{sec:simu2}
In this section, we present a data generation process which produces time-series that behave similarly to road traffic data. 
We then use the output to validate the regime switching approach by comparing it to other baseline algorithms.

\subsection{Data generation process}

\subsubsection*{Model parameters}
 The intercept sequence  $(b_t)_{t=t,\ldots,T}$ is defined as follows.  For each $t\in \{1,\ldots, T\}$ and $k\in \{1,\ldots, p\}$, $(b_{t})_k = -(2.5^2-(t-17.5)^2)$.
The matrix $A $ is constructed as follows.

\begin{itemize} 

  \item Generate an Erd\"os-R\'enyi graph \citep{erdds1959random} with an average of $8$ connections for each node in the network. This results in a sparse structure for the matrix $A$.
  
  \item Generate independently the nonzero entries of the matrix $A $ according to a uniform distribution on $[-1,1]$.
  
  \item Normalize the matrix so that each line has an $\ell_2$-norm equal to $1$. 
  
\end{itemize}

 The last step above helps to obtain reasonable dynamics for the sequences of interest, for which the generation is described below. 

\subsubsection*{Generation of the sequence }

 Let $b$ be an intercept matrix and let $A$ and $A'$ be two matrices constructed as before. For each day $ \{1 \dots n\}$, the initial speed vector $W_0^{(i)}$ of size $p$ has independent entries distributed as a mixture between $3$ Gaussian distributions (to get roughly $3$ different classes of speed) given by $$0.25 \mathcal{N}(S_1,0.1S_1/2) + 0.5\mathcal{N}(S_2,0.1S_2/2) + 0.25 \mathcal{N}(S_3,0.1S_3/2), $$ with $ S_1= 45$, $S_2 = 72$ and $S_3 = 117 $. 
   Generate $(\epsilon_t^{(i)})_{t=1,\ldots T}\subset \mathbb R^p$ as an independent and identically distributed sequence with distribution $\mathcal{N}(0,I_p)$.  The data is then obtained based on 
\begin{align*}
   &W_{t+1}^{(i)}= b_t + A (W_t^{(i)} - \mathbb{E}[W_t^{(i)}]) + \epsilon_t^{(i)}, \qquad t = 1,\ldots, 11,\\
   &W_{t+1}^{(i)}= b_t + A' (W_t^{(i)} - \mathbb{E}[W_t^{(i)}]) + \epsilon_t^{(i)},\qquad t = 12,\ldots, 18 .
\end{align*}

\subsection{Results}
The RS-LASSO approach introduced in Section  \ref{sec:model_switch} results from a search of the best instant $t$ from which a different (linear) model shall be used. This search is conducted by comparing the estimated risks $\hat R_t$, $t=1,\ldots, T$, where $t$ is the instant from which a second model is estimated. 
\begin{figure}[h]
  \centering
  \includegraphics[width=0.8\textwidth]{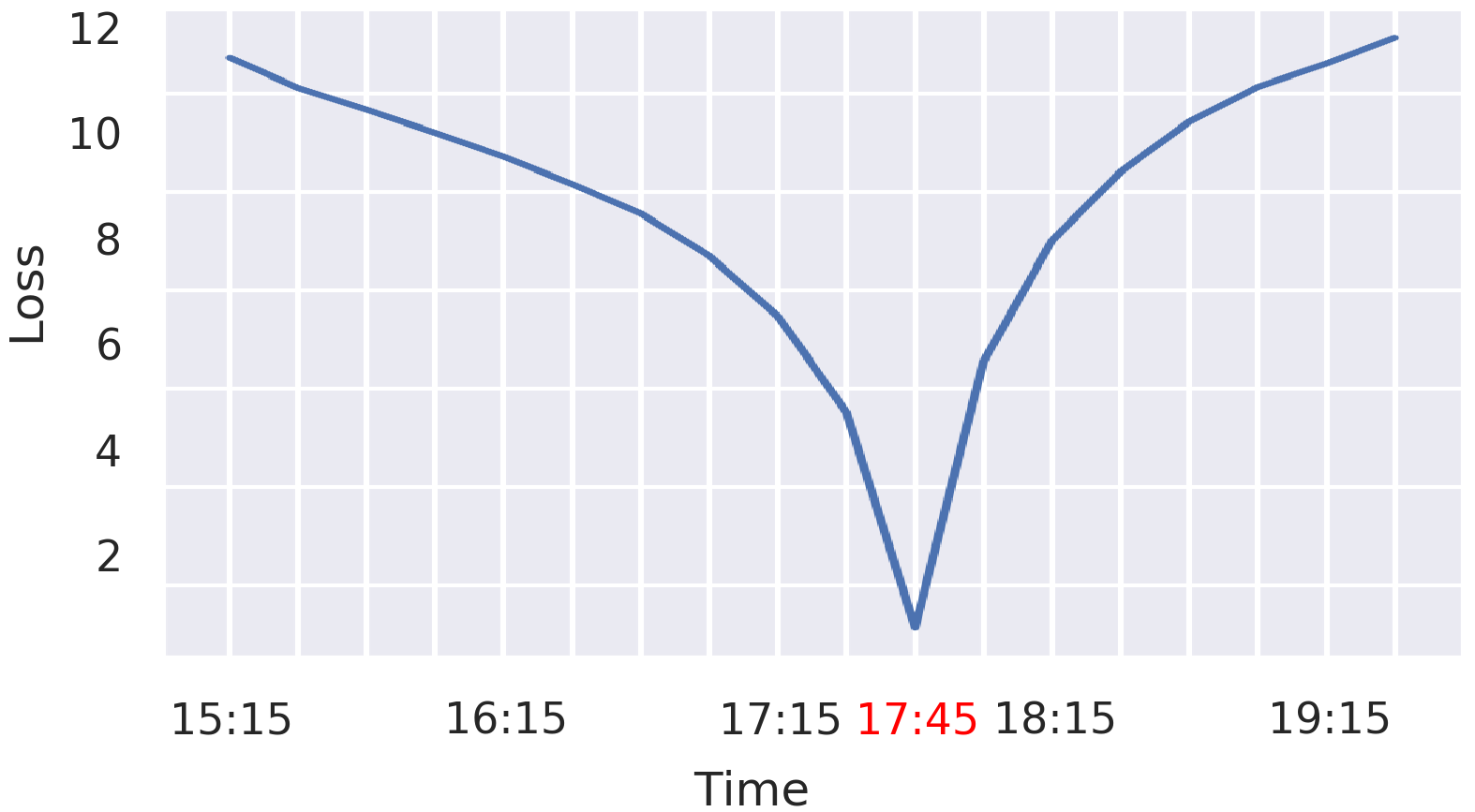}
  \caption{Evolution of the estimated risks $\hat R_t$.}
  \label{fig:rsm_split}
\end{figure}

Figure \ref{fig:rsm_split} represents the evolution of the estimated risks for a single representative dataset.  
The best change point can be easily identified as the procedure gives $\hat t = 11$. 

\begin{table}[h!]
\resizebox{\textwidth}{!}{%
\begin{tabular}{lrrrrrrrrrrrr}
\hline
\textbf{Model} & \textbf{OLS} & \textbf{MLP} & \textbf{FC-2LSTM} & \textbf{AR 1} & \textbf{AR 3} & \textbf{AR 5} & \textbf{HA} & \textbf{PO} & \textbf{LASSO} & \textbf{Grp-LASSO} & \textbf{RS-LASSO} & \textbf{TS-LASSO} \\ \hline
\textbf{MSE}         & 18.94        & 16.13        & 17.43             & 52.55         & 45.11         & 44.9          & 60.32       & 110.68      & 13.97          & 14.81              & \textbf{1.13}     & 5.30              \\ \hline
\textbf{MAE}         & 3.24         & 2.96         & 3.13              & 5.65          & 5.21          & 5.2           & 6.04        & 7.95        & 2.75           & 2.85               & \textbf{0.85}     & 1.75              \\ \hline
\end{tabular}%
}
\caption{MSE and MAE computed over the whole network and the whole time period.}
\label{tab:sim_results}
\end{table}

In Table~\ref{tab:sim_results}, we compare the RS-LASSO to the other methods based on $M=20$ independent datasets (the MSE and MAE have been averaged over the datasets). 
The RS-LASSO model obtains the best performance and clearly outperforms the other methods. The fact that RS-LASSO succeeds in identifying the exact switching time allows for the build-up of considerably good predictors for both time-windows before and after the switch. 
Nonlinear models such as MLP and LSTM obtained slightly lower performances than the LASSO and slightly better than OLS. The baseline methods HA, PO, and AR gave the worst performance on the simulated data (as well as on the real data). In short, we can distinguish between four classes of methods reflecting their level of performance. Baseline methods give the worst results. The one-model methods (LASSO, OLS, Grp-LASSO, MLP, LSTM) give reasonable results. The multiple-model TS-LASSO works better than one-model methods and finally, the regime-switching RS-LASSO performs the best (by fitting only two models to the data).


%
%
%
%
%

\begin{table}[h!]
\resizebox{\textwidth}{!}{%
\begin{tabular}{lrrll}
\hline
            & FD(RS-LASSO) & FD( LASSO) & SR(RS-LASSO) & SR( LASSO) \\ \hline
Left model  & 1.592872                & 19.755836            & 96.49\%                  & 79.11\%               \\ \hline
Right model & 1.301384                & 18.232622            & 96.37\%                  & 79.22\%               \\ \hline
\end{tabular}%
}
\caption{Frobenius distance and support recovery rate for LASSO and RS-LASSO.}
\label{tab:norm_results}
\end{table}

In order to further investigate the RS-LASSO model, we compare the estimated matrices $\hat A $ and $\hat A '$ to the ones used in the data generation process. 
Table \ref{tab:norm_results} shows the Frobenius distance (FD) $\| \hat A - A \|_F+ \| \hat A' - A' \|_F$ along with the support recovery rates $(SR)$ defined as the percentage of accordance in the non-zero coefficients (active set) between the estimated matrix and the truth. 
In contrast with the LASSO, we observe a very small Frobenius error for RS-LASSO. In the meantime, the support recovery of RS-LASSO is more than $96\%$ compared to around $79\%$ for the LASSO. This confirms the favorable performance of RS-LASSO. 

\section{Conclusion}
We proposed in this paper a methodology to model high-dimensional time-series for which regeneration occurs at the end of each day. In a high-dimensional setting, we established bounds on the prediction error associated to linear predictors.  More precisely, we showed that the regularized least-squares approach provides a decrease of the prediction error bound compared to ordinary least-squares. In addition, a regime switching detection procedure has been proposed and its consistency has been established. Through simulations and using floating car data observed in an urban area, the proposed approach has been compared to several other methods. Further research might be conducted on the extension of the proposed approach to general Markov chains. In this case, the regeneration times are unobserved and occur randomly involving a completely different probabilistic setting than the one borrowed in this paper. Whereas regeneration properties are known to be powerful theoretical tools to study estimates based on Markov Chains (see e.g., \citep{bertail2019rademacher} for a recent account), their use in a prediction framework seems not to have been investigated. Another possible extension concerns the identification of multiple change points. On the theoretical side, the consistency of the proposed cross-validation procedure has been established in Section \ref{sec:model_switch} only when a unique change point occurs.
  
%

\paragraph{Acknowledgment} The authors are grateful to Maxime Bourliatoux for helpful discussions on an earlier version of this work and Bethany Cagnol for assistance in proof reading the article. This work was financially supported by the FUI 22 GEOLYTICS project. The authors would like to thank IT4PME and Coyote for providing data and for their support.

\bibliographystyle{apalike}
\bibliography{bib_raod_traffic}

\begin{appendices}

\section{Proofs of the stated results}

Before starting the proofs, we specify some useful notations. For any matrix $M \in \mathbb R^{p\times T}  $, the $\ell_1$-norm is denoted by $\|M\|_1  = \sum_ {k=1}^p\sum_ {t=1}^T |M_{k,t} | $ and the $\ell _\infty$-norm is given by $\|M\|_\infty = \max _{k=1,\ldots, p,\, t= 1,\ldots, T} | M_{k,t} |$. The $t$-th column (resp. $k$-th line) of $M$ is denoted by $M_{\{\cdot, t\}}\in \mathbb R^p $ (resp $M_{\{k,\cdot\}} \in \mathbb R^T$). Both are column vectors.

\subsection{Proof of Proposition \ref{prop:optimal_matrix}}
Denote by $L_2(P)$ the Hilbert space composed of the random elements $W$ defined on $(\Omega, \mathcal F,P)$ such that $E \|W\|_F<+\infty$. The underlying scalar product between $X$ and $Y$ is $\tr(E [X^T Y])$. Hence $\inf_{L\in \mathcal L} E [\| Y- L(X) \|_F^2 ] $ is a distance between the element $Y$ of $L_2(P)$ and the linear subspace of $L_2(P)$ made of $ L(X)$, $L\in \mathcal L$. The Hilbert projection theorem ensures that the infimum is uniquely achieved and that $L^\star(X)$ is the argument of the minimum if and only if $  \tr(E [( Y - L^\star (X) )^T L(X) ]) = 0 $ for all $L  \in \mathcal L$. Taking $A=0$, in $L( X ) = b + A X$, gives the first set of normal equations. We conclude by taking $b= 0$ and  noting that $\tr(E [( Y - L^\star (X) )^T L(X) ]) = \tr(E [ X ( Y - L^\star (X) )^T  ] A )  = 0 $ for all $A\in \mathbb R^{p\times p}$ is equivalent to  $E[ (  Y-  L^\star( X) ) X^T ]  = 0$. This is the second set of normal equations. 
\qed

\subsection{Proof of Proposition \ref{prop:risk_decomp}}
The  statement follows from developing $p R(L) =  E [\| (Y-L^\star (X) ) + (L^\star (X) - L(X) )    \|^2] $ and then using that $\tr ( E [( Y -  L^\star(X) )  ( L(X)- L^\star (X))^T] )  = 0$ which is a consequence of the normal equations given in Proposition \ref{prop:uniq_min_homo}. 
\qed

\subsubsection{Proof of Proposition \ref{first_decomp_ols}}

Let $X_c = X - E[X]$. Recall that $\hat L  (X) = \hat b + \hat A  X  $ and  $L^\star (X) = E[Y]  + A^\star X_c$. In virtue of Proposition \ref{prop:risk_decomp}, we have
\begin{align*}
p \mathcal E (\hat L) &=     E_X [\| (A^\star -\hat A  )X_c + E(Y) - \hat AE(X) - \hat b \|_F^2] ,
\end{align*}
where $E_X$ stands for the expectation with respect to $X$ only. Because, $X_c$ has mean $0$, we get
\begin{align*}
p \mathcal E (\hat L) &=  E_X [\| (A^\star -\hat A ) X_c \|_F^2] +  \| E(Y) - \hat A E(X) - \hat b \|_F^2 .
\end{align*}
Finally, since $  E_X [\| (A^\star -\hat A  ) X_c \|_F^2]  =  \| (A^\star - \hat A ) \Sigma^{1/2} \|_F^2$,  and using the triangle inequality, we find
\begin{align*}
p \mathcal E (\hat L)  &\leq \| (A^\star -\hat A ) \Sigma^{1/2} \|_F^2 +  2\| E(Y) - \overline Y^n \|_F^2 + 2\| \hat A (E(X) - \overline X^n) \|_F^2.
\end{align*}
\qed

\subsubsection{Proof of Proposition \ref{prop:ols}}

The  inequality of Proposition \ref{first_decomp_ols} applied to the OLS estimate gives
\begin{align*}
p \mathcal E (\hat L^{(ols)}) \leq  \| ( A^\star -  \hat A ^{(ols)} ) \Sigma^{1/2} \|_F^2 +  2\| \overline{Y}^n -E (Y) \|_F^2 + 2 \|  \hat A^{(ols)}  \Sigma^{1/2} \overline {Z}^n  \|_F^2.
\end{align*}
Using that $\|  \hat A^{(ols)} \Sigma^{1/2} \overline {Z}^n  \|_F \leq \| ( \hat A^{(ols)}-  A^\star) \Sigma^{1/2}\|_F \|\overline {Z}^n \|_F  + \|  A^\star  \Sigma^{1/2} \overline {Z}^n  \|_F$, we get
\begin{align*}
p \mathcal E (\hat L^{(ols)}) \leq  \|  (A^\star -  \hat A ^{(ols)})\Sigma^{1/2} \|_F^2 (1 + 4\|\overline {Z}^n \|_F^2)    +  2\| \overline {Y}^n - E(Y) \|_F^2 &\\
+ 4   \|  A^\star \Sigma^{1/2} \overline {Z}^n  \|_F^2 .&
\end{align*}
Let $\Gamma $ denote the greatest eigenvalue of $\Sigma$. By Assumption \ref{ass_4}, we have $  \limsup_{n\to \infty}  \Gamma < \infty$. Applying Lemma \ref{lemma:basic_stuff}, we deduce that
\begin{align}
\nonumber p\mathcal E (\hat L^{(ols)}) & =  \|  (A^\star -  \hat A ^{(ols)})\Sigma^{1/2} \|_F^2 \left( 1 + O_P\left(  \frac{p }{ n} \right)\right )     +  O_P\left(  \frac{p  + \|A^\star \Sigma^{1/2}   \|_F^2}{  n} \right )  \\
\nonumber &=    \|  (A^\star -  \hat A ^{(ols)})\Sigma^{1/2} \|_F^2O_P(1)    +  O_P\left(  \frac{p  + \|A^\star \Sigma^{1/2}  \|_F^2}{  n} \right )  \\
\label{first_decomp_ols_2} &=    \|  (A^\star -  \hat A ^{(ols)})\Sigma^{1/2} \|_F^2O_P(1)    +  O_P\left(  \frac{p  + \Gamma \|A^\star\|_F^2 }{  n} \right ) ,
\end{align}
where we just used that $p/n \to 0$ and $\|A^\star \Sigma^{1/2}  \|_F^2 = \sum_{k=1}^p \|  A^{\star T}_{\{k,\cdot \}}  \Sigma^{1/2} \|_2^2\leq \Gamma \|A^\star\|_F^2$. Hence the proof of Proposition \ref{prop:ols} will be complete if it is shown that
\begin{align*}
\|  (\hat A ^{(ols)} -  A^\star )\Sigma^{1/2}  \|_F^2 = O_P \left(  {\frac{ p^2\sigma^2  }{ n} }  \right)  .
\end{align*}
By definition of $\hat A^{(ols)}$, it holds that 
\begin{align*}
\sum_{i=1}^n \| \hat Y_c^{(i)}  -   \hat A ^{(ols)} \Sigma^{1/2}   \hat Z_c^{(i)} \|_F^2 \leq \sum_{i=1}^n  \| \hat Y_c^{(i)}  -  A^\star  \Sigma^{1/2}   \hat Z_c^{(i)} \|_F^2.
\end{align*}
where  
\begin{align*}
&\hat Z_{c}^{(i)} =  \Sigma^{-1/2 } ( X_{}^{(i)} - \overline{X}^n) \qquad \text{and} \qquad  \hat Y_c^{(i)} = Y^{(i)} - \overline{Y}^n .
\end{align*}
 Let $\hat \Delta^{(ols)} = ( A^\star -  \hat A ^{(ols)} ) \Sigma^{1/2} $. Developing the squared-norm in the left-hand side and making use of the centering property, we obtain
\begin{align}
\sum_{i=1}^n  \|  \hat \Delta^{(ols)} \hat Z^{(i)} \|_F^2 &\leq 2 \left| \sum_{i=1}^n  \langle \hat Y_c^{(i)}  -  A^\star \Sigma^{1/2}   \hat Z^{(i)} , \hat \Delta^{(ols)}   \hat Z^{(i)} \rangle_F\right|
\label{eq:imp_ols} = 2 \left| \sum_{i=1}^n  \langle  \epsilon^{(i)} - \overline{\epsilon}^n , \hat \Delta^{(ols)}   \hat Z^{(i)} \rangle_F\right|,
\end{align}
with 
\begin{align*}
\epsilon^{(i)} = Y^{(i)} - L^\star (X^{(i)}).
\end{align*}
We now provide a lower bound for the left-hand side of \eqref{eq:imp_ols}. Define
\begin{align*}
&\hat \Pi =  \Sigma^{1/2} \hat \Sigma \Sigma^{1/2},\\
&\hat \Sigma  = n^{-1} \sum_{i=1} ^n  (X^{(i)} - \overline{X}^n)  (X ^{(i)} - \overline{X}^n) ^T.
\end{align*}
We have $n^{-1} \sum_{i=1}^n  \| \hat \Delta^{(ols)}  \hat Z^{(i)} \|_F^2  = \tr  ( \hat \Delta^{(ols)}  \hat \Pi \hat \Delta^{(ols)T}  ) $, and because $ \hat \Pi$ is symmetric we can write $  \hat \Pi  = UDU^T$ with $U^T U =I_p$ ($I_p$ is the identity matrix of size $p\times p$) and $D = \text{diag}(d_1,\ldots, d_p)$. Hence, defining $\tilde \Delta^{(ols)} = (\tilde \Delta^{(ols)} _{\{\cdot,1\}}, \ldots, \tilde \Delta^{(ols)}_{\{\cdot,p\}} ) = \hat \Delta^{(ols)} U $, it holds that
\begin{align*}
n^{-1}  \sum_{i=1}^n  \| \hat \Delta^{(ols)}  \hat Z^{(i)}  \|_F^2  = \tr \left(\sum_{k=1} ^p  d_k \tilde \Delta_{\{\cdot,k\}}^{(ols)} \tilde \Delta_{\{\cdot,k\}}^{(ols)T} \right) &\ge \gamma_{n} \sum_{k=1} ^p  \| \tilde \Delta_{\{\cdot,k\}}^{(ols)}\|_2^2  \\
&= \gamma_n\| \hat \Delta^{(ols)} \|_F^2 . 
\end{align*}
For the right-hand side of \eqref{eq:imp_ols}, we provide the following upper bound. Using the Cauchy-Schwarz inequality, we have
 \begin{align*}
\left| \sum_{i=1}^n  \langle \epsilon^{(i)} - \overline{\epsilon}^n , \hat \Delta^{(ols)}   \hat Z^{(i)} \rangle\right|  
&=   \left|   \langle   \sum_{i=1}^n (\epsilon^{(i)} - \overline{\epsilon}^n )\hat Z^{(i)T}  , \hat \Delta^{(ols)}   \rangle\right|\\
&\le \|  \sum_{i=1}^n (\epsilon^{(i)} - \overline{\epsilon}^n )\hat Z^{(i)T}  \|_F \|\hat \Delta^{(ols)}\|_F.
\end{align*}
Together with \eqref{eq:imp_ols}, the two previous bounds imply that
\begin{align*}
\gamma_n  \|   \hat \Delta^{(ols)} \|_F \le 2  \| n^{-1}  \sum_{i=1}^n  (\epsilon^{(i)} - \overline{\epsilon}^n ) \hat Z ^{(i)T}  \|_F.
\end{align*} 
Note that, for all $u\in \mathbb R^p$ with $\|u\|_2 =1$,
\begin{align*}
 u^T \hat \Pi u = 1  +  u^T  ( \hat \Pi - I) u \geq 1 - \| \hat \Pi - I\| .
\end{align*}
where $\|\cdot\| $ stands for the spectral norm. Hence we have that
\begin{align*}
\left(1 - \| I-\hat \Pi\|\right) \|   \hat \Delta^{(ols)} \|_F \le 2  \| n^{-1}    \sum_{i=1}^n  (\epsilon^{(i)} - \overline{\epsilon}^n ) \hat Z ^{(i)T}  \|_F.
\end{align*} 
Because the term in the right hand-side is $O_{ P}(p\sigma / \sqrt n  ) $ in virtue of  Lemma \ref{lemma:residuals_frob},  we just have to show that $ \| \hat \Pi - I\|\to 0$ in probability to conclude the proof. For that 
purpose we use Lemma  \ref{lemma:eigenvalues} with $B$ given by the following inequality (which holds true in virtue of Assumptions \ref{ass_4} and \ref{ass_3}),
\begin{align*}
 \tr \left( X_c ^{T} \Sigma ^ {-1}  X_c \right) &= \sum_{t=1} ^T (X_{\{\cdot, t\}} - E (X_{\{\cdot, t\}}) ^T  \Sigma ^ {-1} (X_{\{\cdot, t\}} - E (X_{\{\cdot, t\}}) \\
 & \leq  \sum_{t=1} ^T \gamma^{-1}  \| X_{\{\cdot, t\}} - E (X_{\{\cdot, t\}} \|^2_2  \\
 & \leq TpU^2 /\gamma,
\end{align*}
where, almost surely, $U = \limsup_{n \to \infty} \max _{k\in S, \, t\in \mathcal T} | X_{k,t} - E(X_{k,t} )  |$. Because $ p\log(p) / (n\gamma )\to 0$, the upper bound given in Lemma \ref{lemma:eigenvalues} goes to $0$ and it holds that $\| \hat \Pi - I\| \to 0 $, in probability.

\qed



%

\subsubsection{Proof of Proposition \ref{prop:Lasso}}
Define $ \hat \Delta =  (A^\star -  \hat A ^{(lasso)})$. A similar decomposition to \eqref{first_decomp_ols_2} is valid for the LASSO. We have
\begin{align*}
p\mathcal E (\hat L^{(lasso)}) \leq  \| \hat \Delta \Sigma^{1/2} \|_F^2   O_P(1)    +  O_P\left(  \frac{p  + \Gamma \|A^\star\|_F^2 }{  n} \right )  ,
\end{align*}
implying that the proof reduces to $\| \hat \Delta  \Sigma^{1/2}  \|_F^2 = O_P ( (p/n)  \log(p))$. To this aim, we will provide a bound on $ \|\Sigma^{1/2} \hat \Delta_{\{k,\cdot\}} \|_2 $ for each value of $k = 1,\ldots, p$. The minimization of the LASSO problem can be done by solving $p$ minimization problems where each gives one of the lines of the estimated matrix $ \hat A ^{(lasso)}  = ( \hat A_{\{1,\cdot\}}^{(lasso)} , \ldots,  \hat A_{\{p,\cdot\}}^{(lasso)})^T$. More formally, we have
\begin{align*}
  \hat A_{\{k,\cdot\}}^{(lasso)}    \in \argmin_{\beta \in \mathbb R^{p } }  \sum_{i=1}^n  \| (Y_{k}^{(i)} - \overline{Y_k}^n)  -  (X^{(i)} - \overline{X}^n)^T  \beta \|_2^2 +  2\lambda \|\beta\| _1.
\end{align*}
Denote by $  s_\infty ^\star = \limsup_{n\to \infty }\max_{k\in \mathcal{S}} |\mathcal S_k^\star|   $ and let $U $ be such that, almost surely,
\begin{align*}
& \limsup_{n\to \infty } \max _{k\in \mathcal S, \, t\in \mathcal T} | X_{k,t} - E(X_{k,t} )  |\leq U\qquad \text{and} \qquad  \limsup_{n\to \infty } \max _{k\in \mathcal S, \, t\in \mathcal T} |\epsilon_{k,t} |\leq U.
\end{align*}
Such a $U$ exists in virtue of Assumption \ref{ass_4} and \ref{sparsity_level}. Set
\begin{align*}
 & \lambda = 8\sqrt{  T^2 U^2 \sigma^ 2 n \log( p^3   )}.
 \end{align*}
Intermediary statements, that are useful in the subsequent development, are now claimed and will be proved at the end of the proof. Let $0<\delta<1$. For $n$ large enough, we have with probability $1-\delta/2$, for all $k\in \{1,\ldots, p\}$,
\begin{align}
\label{Lasso_cone_1} &n\| \hat \Sigma^{1/2}  \hat \Delta_{\{k,\cdot\}}\|_2^2 \leq  \lambda  (  3 \|\hat \Delta_{\{k,{S}_k^\star\}}\| _1 -  \|\hat \Delta_{\{k, {S}^{\star c}_k\}}\|_1 ),
\end{align}
In particular, $\hat \Delta_{\{k,\cdot\}} \in \mathcal C( 3 , S ^ \star_k ) $. Moreover, for $n$ large enough, with probability $1-\delta/2$, 
\begin{align}
\label{Lasso_invert_prop2} 16 \|  \hat \Sigma -  \Sigma   \|_\infty  (s^\star_\infty / \gamma ^\star) \leq 1 / 2 .
\end{align}
In the following, we assume that \eqref{Lasso_cone_1} and \eqref{Lasso_invert_prop2} are satisfied. This occurs with probability $1-\delta$.   From \eqref{Lasso_cone_1}, we have $\| \hat \Delta_{k}\|_1= \| \hat \Delta_{k,{S}_k^\star }\|_1 + \| \hat \Delta_{k,  {S}_k^{\star c}  } \|_1  \leq  4 \| \hat \Delta_{k, {S}_k^\star}\|_1 $. Then, using Jensen inequality, Assumption \ref{ass:RE}, and finally \eqref{Lasso_invert_prop2}, it follows that
\begin{align*}
|\| \Sigma^{1/2}  \hat \Delta_{\{k,\cdot\}} \|_2^2 - \| \hat \Sigma^{1/2}  \hat \Delta_{\{k,\cdot\}}\|_2^2 |& =   | \hat \Delta_{\{k,\cdot\}} ^T (  \hat \Sigma -  \Sigma )  \hat \Delta_{\{k,\cdot\}} | \\
& \leq  \|  \hat \Sigma -  \Sigma   \|_\infty \| \hat \Delta_{\{k,\cdot\}}\|_1^2  \\
& \leq     16 \|  \hat \Sigma -  \Sigma   \|_\infty   \| \hat \Delta_{\{k,{S}_k^\star\}}\|_1^2 \\
&\leq    16 \|  \hat \Sigma -  \Sigma   \|_\infty  |{S}^\star_k| \| \hat \Delta_{\{k,{S}_k^\star\}}\|_2^2\\
&\leq  16 \|  \hat \Sigma -  \Sigma   \|_\infty  |{S}^\star_k| \| \hat \Delta_{\{k,\cdot\}}\|_2^2\\
&\leq  16 \|  \hat \Sigma -  \Sigma   \|_\infty ( |{S}^\star_k| / \gamma ^{\star}) \| \Sigma^{1/2} \hat \Delta_{\{k,\cdot\}}\|_2^2\\
&\leq  \| \Sigma^{1/2} \hat \Delta_{\{k,\cdot\}}\|_2^2 / 2 .
\end{align*}
In particular 
\begin{align}\label{ineq:sigma_n}
 \| \hat \Sigma^{1/2}  \hat \Delta_{\{k,\cdot\}} \|_2^2 & \leq 2 \| \Sigma^{1/2} \hat \Delta_{\{k,\cdot\}}\|_2^2  .
\end{align}
%
Invoking Assumption \ref{ass:RE} we find 
\begin{align*}
\|\hat \Delta_{\{k,\cdot\}} \|_2^2  &\leq   \|\Sigma^{1/2}  \hat \Delta_{\{k,\cdot\}} \|_2^2 / \gamma^\star   \leq  2\|\hat \Sigma^{1/2}  \hat \Delta_{\{k,\cdot\}} \|_2^2 / \gamma^\star  ,
\end{align*}
which in turn gives
\begin{align}\label{ineq:l1}
 \|\hat \Delta_{\{k,{S}_k^\star\}}\|_1 \leq \sqrt {|{S}_k^\star|} \|\hat \Delta_{\{k,{S}_k^\star\}} \|_2
 \leq \sqrt {|{S}_k^\star|} \|\hat \Delta_{\{k,\cdot\}} \|_2  &\leq \sqrt { 2 (|{S}_k^\star| / \gamma^\star ) } \|\hat \Sigma^{1/2}  \hat \Delta_{\{k,\cdot\}} \|_2 .
\end{align}
Injecting this into the right-hand side of \eqref{Lasso_cone_1}, we get
\begin{align*}
n\| \hat \Sigma^{1/2}  \hat \Delta_{\{k,\cdot\}}\|_2^2\leq
3 \lambda   \sqrt { 2\frac{|{S}_k^\star|}  { \gamma^\star}} \| \hat \Sigma^{1/2}  \hat \Delta_{\{k,\cdot\}}\|_2,
\end{align*}
which implies that 
\begin{align*}
& n\| \hat \Sigma^{1/2}  \hat \Delta_{\{k,\cdot\}}\|_2 \leq  3\lambda \sqrt  { 2 \frac{|{S}_k^\star|}  { \gamma^\star}} .
\end{align*}
We conclude the proof using \eqref{ineq:sigma_n} to obtain
\begin{align*}
 \| \hat \Delta  \Sigma^{1/2}  \|_F^2 = \sum_{k=1}^p  \| \Sigma^{1/2} \hat \Delta_{\{k,\cdot\}}  \|_2^2 \leq  \frac{36 \lambda^2 }{n^2 \gamma^{\star}} \sum_{k=1} ^p  {|{S}_k^\star|}  
 =  \frac{ Cp \sigma^ 2       \log( p  ) }{n}   ,
\end{align*}
for some $C>0$ that does not depend on $(n,p)$, and recalling that the previous happens with probability $1-\delta$ with $\delta $ arbitrarily small.

\noindent\textbf{Proof of Equation \eqref{Lasso_cone_1} and \eqref{Lasso_invert_prop2}.} As $\log(p) / n \to \infty$ and $\liminf_{p\to \infty} \sigma^2>0$, we have, for $n$ large enough,
$ n \geq 4 (U^2 / \sigma^2) \log( 12p^2 / \delta  )   \geq   \log(8p^2 / \delta ) $. Note that the last inequality implies that we can apply Lemma  \ref{lemma:unif_epsilon} and \ref{lemma:unif_sigma}.

\noindent\textbf{Proof of Equation \eqref{Lasso_cone_1}.} Recall that $ \lambda = 4\sqrt{  T^2 U^2 \sigma^ 2 n \log( 12 p^3   )}$. Note that for $p$ large enough,
 \begin{align*}
 & \lambda \geq  4\sqrt{  T^2 U^2 \sigma^ 2 n \log( 12 p^2 / \delta  )}.
\end{align*}
It follows from Lemma \ref{lemma:unif_epsilon} that with probability $1-\delta/2$:
\begin{align}\label{eq:proof_Lasso_bound1}
\lambda \geq  2 \left\|\sum_{i=1}^n (\epsilon ^{(i)} - \overline{\epsilon}^n )  (X^{(i)} - \overline{X}^n )^T  \right\|_\infty .
\end{align}
Let $k\in \{1,\ldots, p\}$. Note that because of \eqref{eq:proof_Lasso_bound1}, it holds that
\begin{align*}
\lambda \geq  2 \left\|\sum_{i=1}^n (\epsilon_{\{k,\cdot\}} ^{(i)} - \overline{\epsilon_{\{k,\cdot\}}}^n )  (X^{(i)} - \overline{X}^n )^T  \right\|_\infty .
\end{align*}
We have
 \begin{align*}
&\frac{1}{2} \sum_{i=1}^n  \|   (X ^{(i)} - \overline{X} ^n )^T \hat \Delta_{\{k,\cdot\}}\|_2^2 \\
&\leq  \left| \sum_{i=1}^n  \langle \epsilon_{\{k,\cdot\}} ^{(i)} - \overline{\epsilon_{\{k,\cdot\}}}^n   ,   (X ^{(i)} - \overline{X} ^n )^T \hat \Delta_{\{k,\cdot\}}\rangle\right| + \lambda ( \|A_{\{k,\cdot\}}^\star\|_1 - \|\hat A_{\{k,\cdot\}}\|_1 ) .
\end{align*}
First consider the scalar product term of the right-hand side. We have
\begin{align*}
&\left| \sum_{i=1}^n  \langle \epsilon_{\{k,\cdot\}} ^{(i)} ,   (X ^{(i)} - \overline{X} ^n )^T \hat \Delta_{\{k,\cdot\}}\rangle \right| \\
 &= \left|   \langle  \sum_{i=1}^n   (X ^{(i)} - \overline{X} ^n ) (\epsilon_{\{k,\cdot\}} ^{(i)} - \overline{\epsilon_{\{k,\cdot\}}}^n ) , \hat \Delta_{\{k,\cdot\}}   \rangle_F\right|\\
&\leq  \|\sum_{i=1}^n  (X ^{(i)} - \overline{X} ^n ) (\epsilon_k ^{(i)} -  \overline{\epsilon_{\{k,\cdot\}}}^n )  \|_\infty \|\hat \Delta_{\{k,\cdot\}} \|_1\\
& \leq \|\sum_{i=1}^n   (\epsilon ^{(i)}_{\{k,\cdot\}} - \overline{\epsilon_{\{k,\cdot\}}}^n  )(X ^{(i)} - \overline{X} ^n )^T  \|_\infty\|\hat \Delta_{\{k,\cdot\}} \|_1\\
&\leq (\lambda/2)    \|\hat \Delta_{\{k,\cdot\}} \|_1.
\end{align*}
Now we deal with $\|A^\star_{\{k,\cdot\}}\|_1 - \|\hat A_{\{k,\cdot\}}\|_1$.  Note that, by the triangle inequality, 
\begin{align*}
 \|\hat A_{\{k,\cdot\}}\|_1  = \| A^\star_{\{k,\cdot\}} - \hat \Delta_{\{k,\cdot\}} \| _ 1&\geq | \|A^\star_{\{k,\cdot\}} \|_1 -  \|\hat \Delta_{\{k,\cdot\}}\|_1 |\\
 &\geq \|A^\star _{\{k,\cdot\}}\|_1 - \|\hat \Delta_{\{k,\cdot\}}\| _1 .   
\end{align*}
Using the previous, restricted to the active set ${S}^\star_k$, we obtain
\begin{align*}
\|A^\star_{\{k,\cdot\}}\|_1 - \|\hat A_{\{k,\cdot\}}\|_1 &= \|A^\star_{\{k,\cdot\}}\|_1 - \|\hat A_{k,{S}^\star_k}\|_1 -  \|\hat A_{k,{S}^{\star c}_k}\|_1\\
&\leq \|A^\star_{\{k,\cdot\}}\|_1 - (\|A^\star_{{S}^\star_k} \|_1 - \|\hat \Delta_{\{k,{S}^\star_k\}}\| _1) -  \|\hat A_{k,{S}^{\star c}_k}\|_1\\
& =  \|\hat \Delta_{\{k,{S}^\star_k\}}\| _1 -  \|\hat \Delta_{\{k,  {S}^{\star c}_k\}}\|_1.
\end{align*}
All this together gives
\begin{align*}
\frac{1}{2} \sum_{i=1}^n  \|   (X ^{(i)} - \overline{X} ^n )^T \hat \Delta_{\{k,\cdot\}} \|_2^2 &\leq (\lambda /2) ( \|\hat \Delta_{\{k,\cdot\}}\|_1 +2 \|\hat \Delta_{\{k,{S}^\star_k\}}\| _1 - 2 \|\hat \Delta_{\{k,  {S}^{\star c}_k\}}\|_1 )\\
&= (\lambda /2) (  3 \|\hat \Delta_{\{k,{S}^\star_k\}}\| _1 -  \|\hat \Delta_{\{k,  {S}^{\star c}_k\}}\|_1 ),
\end{align*}
and \eqref{Lasso_cone_1} follows from remarking that
\begin{align*}
\sum_{i=1}^n  \|  \hat \Delta_{\{k,\cdot\}}^T(X ^{(i)} - \overline{X} ^n )  \|_2^2 = n \| \hat \Sigma ^{1/2} \hat \Delta_{\{k,\cdot\}}  \|_2^2.
\end{align*}

%

\noindent\textbf{Proof of Equation \eqref{Lasso_invert_prop2}.} Use that for $n$ sufficiently large, we have $ \sqrt n \geq 16^2 (  s^\star_\infty /\gamma^{\star }  )  T U^2 \sqrt{  \log( 8p^2/ \delta )  }  $ and  Lemma \ref{lemma:unif_sigma} which guarantees that $\| \hat \Sigma - \Sigma \|_\infty \leq  8 T U^2 \sqrt{  \log( 8 p^2/ \delta )  /n }$.

\qed

\subsection{Proof of Proposition \ref{prop:regime_switch1}}
The proof is done under a simplified setting which does not involve any loss of generality. Let $t\in \{1,\ldots ,T\} $. Since the number of folds  $K$ is fixed, we only need to show the convergence of the risk estimate based on each fold i.e., that $|\hat R _{I,t} - R(L ^\star_t) |  = O_P (  \sqrt { \log(p) / n }   )$ for $I\in \mathcal F$. Moreover both quantities $\hat R _{I,t}$ and $R(L_t^\star) $ are made of two similar terms, one for each period $U_t$ and $V_t$. We focus only on the first one, the details for the other being the same. Moreover since the index $t$ is fixed in the whole proof, we remove it from the notation. We note  $\hat b_{J }, \hat A_{J} , U$ instead of $\hat b_{J, t},  \hat A_{J,t}, U_t$ (recall that $J$ is the complement of $I $ in $\{1,\ldots, n\}$). Define 
\begin{align*}
 \hat \delta  : =\frac{K}{ np  }  \left\{    \sum_{i\in  I } \| Y_{\{\cdot,U \}}^{(i)}   -\hat b_{J}-   \hat A_{J}   X_{\{\cdot,U\}}^{(i)} \|_F^2\right\} &\\
 -   p^{-1} E [ \| Y_{\{\cdot,U\}}^{(1)}   -  b^\star- A^\star   X_{\{\cdot,U\}}^{(1)} \|_F^2] &,
 \end{align*}
with $(b^\star , A^\star) = \argmin _ {b\in \mathbb R^{p\times t} ,\,  A \in \mathbb R^{p\times p} }  E [ \| Y_{\{\cdot,U\}}^{(1)} - b  -  A   X_{\{\cdot,U\}}^{(1)} \|_F^2] $. We need to show that $\hat \delta=
  O_P(\sqrt{  \log(p) /  n  })$. Write 
\begin{align*}
 \hat \delta   &\leq \left| \frac{K}{ np }  \sum_{i\in I}    \{ \| Y_{\{\cdot,U\}}^{(i)}   -\hat b_{J} - \hat A_{J}   X_{\{\cdot,U\}}^{(i)} \|_F^2  -  E_1 [ \|Y_{\{\cdot,U\}}^{(1)} -\hat b_{J}  -  \hat A_{J}   X_{\{\cdot,U\}}^{(1)} \|_F^2 ]  \} \right| \\
&\qquad +  p^{-1} \left \{ E_1 [ \|Y_{\{\cdot,U\}}^{(1)} -\hat b_{J} -  \hat A_{J}   X_{\{\cdot,U\}}^{(1)} \|_F^2 ] - E [ \| Y_{\{\cdot,U\}}^{(1)}  -b^\star -  A^\star   X_{\{\cdot,U\}}^{(1)} \|_F^2] \right \} \\
& : = \hat \delta_{1} + \hat \delta_{2}, 
\end{align*}
where the introduced expectation $E_1$ is taken with respect to $(Y^{(1)}, X^{(1)})$ only. In $\hat \delta_{2} $, we recover the excess risk for the LASSO computed with $n/K$ observations. Hence,  $\hat \delta_{2}  = O_P( \log(p) / n) $ as indicated by Proposition \ref{prop:Lasso}. We now focus on $\hat \delta_{1}$.  Let $ L ^{\star}: x \mapsto b^\star + A^\star x$ be the optimal predictor as introduced in Proposition \ref{prop:optimal_matrix} (with respect to the time period $U$). Put $\hat  L_{J} (  x)  = \hat b_{J} +  \hat A_{J} x $ and $  \epsilon_{\{\cdot ,U\}}^{(i)}  = Y_{\{\cdot,U\}}^{(i)}  - L^{\star}  (  X_{\{\cdot,U\}}^{(i)} )  $. Based on Proposition \ref{prop:optimal_matrix}, we deduce that
\begin{align*}
E_1 [ \|Y_{\{\cdot,U\}}^{(1)} -  \hat  L_{J} ( X_{\{\cdot,U\}}^{(1)} )  \|_F^2 ] &= E_1 [ \| \epsilon_{\{\cdot ,U\}}^{(1)}  \|_F^2 ] +  E_1 [ \| (\hat  L_{J} -   L ^{\star} ) ( X_{\{\cdot,U\}}^{(1)} )  \|_F^2 ]. 
\end{align*}
Similarly,
\begin{align*}
&\frac{K}{ n }  \sum_{i\in I}     \|Y_{\{\cdot,U\}}^{(i)} -  \hat  L_{J} ( X_{\{\cdot,U\}}^{(i)} )  \|_F^2 \\
 &= \frac{K}{ n }  \sum_{i\in I}   \| \epsilon_{\{\cdot ,U\}}^{(i)}  \|_F^2  - \frac{2K}{ n }  \sum_{i\in I}  \tr \left(  \epsilon_{\{\cdot ,U\}}^{(i)T}   (\hat  L_{J} -   L ^{\star} ) ( X_{\{\cdot,U\}}^{(i)} )  \right)   \\
& \qquad + \frac{K}{ n }  \sum_{i\in I}   \| (\hat  L_{J} -   L ^{\star} ) ( X_{\{\cdot,U\}}^{(i)} )  \|_F^2.  
\end{align*}
Let $\hat \Delta _{J} = \hat A_{J} -  A^\star $. Note that 
\begin{align*}
 (\hat  L_{J}  - L ^{\star}  ) (  x ) &= \hat b_{J} - b^\star +  \hat \Delta _{J}  x \\
 & =\overline Y^{(J)}_{\{\cdot,U\} }   -     \hat A_{J} \overline X^{(J)}_{\{\cdot,U\} }  - E_1 Y ^{(1)} _{\{\cdot,U\} }  +       A^\star E_1 X^{(1)}_{\{\cdot,U\} }  +  \hat \Delta _{J}  x\\
 & =\overline Y^{(J)}_{\{\cdot,U\} }  - E_1 Y ^{(1)} _{\{\cdot,U\} }   -     \hat A_{J} (\overline X^{(J)}_{\{\cdot,U\} } - E_1 X^{(1)}_{\{\cdot,U\} } )    +   \hat \Delta _{J}  (  x - E_1 X^{(1)}_{\{\cdot,U\} }  ) \\ 
&  =  \hat \mu_Y^{(J)}   - \hat A_{J} \hat \mu_X^{(J)}  +  \hat \Delta _{J}   (x - E_1 X^{(1)}_{\{\cdot,U\} } ) ,
\end{align*}
with $\hat \mu_Y^{(J)} = \overline Y^{(J)}_{\{\cdot,U\}}    - E_1 Y ^{(1)} _{\{\cdot,U\}} $ and $\hat \mu_X^{(J)} =   \overline X^{(J)}_{\{\cdot,U\}}  - E_1 X^{(1)}_{\{\cdot,U\}}  $. It holds that 
\begin{align*}
\frac{2K}{ n }  \sum_{i\in I}  \tr \left(  \epsilon_{\{\cdot ,U\}}^{(i)T}  (\hat  L_{J} -   L ^{\star} ) ( X_{\{\cdot,U\}}^{(i)} ) \right)  &= \frac{2K}{ n }  \sum_{i\in I}   \tr \left(   (\hat  L_{J} -   L ^{\star} ) ( X_{\{\cdot,U\}}^{(i)} )  \epsilon_{\{\cdot ,U\}}^{(i)T}    \right)  \\
&= 2\tr \left( (\hat \mu_Y^{(J)}  - \hat A_{J} \hat \mu_X^{(J)})    \hat \mu_\epsilon^{( I)}    \right)  + 2\tr \left( \hat \Delta _{J}   \hat \mu_{\epsilon X}  ^{(I)}      \right), 
\end{align*}
with $ \hat \mu_\epsilon ^{(I)} = (K/  n  )   \sum_{i\in I} \epsilon_{\{\cdot ,U\}}^{(i)T} $ and $\hat \mu_{\epsilon X}  ^{(I)}  = ( K /  n  )   \sum_{i\in I}  ( X_{\{\cdot,U\}}^{(i)}  -E_1 X^{(1)}_{\{\cdot,U\}} )\epsilon_{\{\cdot ,U\}}^{(i)T}    $.
In addition, we have
\begin{align*}
&\| (\hat  L_{J} -   L ^{\star} ) ( X_{\{\cdot,U\}}^{(i)} )  \|_F^2   - E_1 \| (\hat  L_{J} -   L ^{\star} ) ( X_{\{\cdot,U\}}^{(1)} )  \|_F^2  \\
&=  2\tr\left( (\hat \mu_Y^{(J)}   - \hat A_{J} \hat \mu_X^{(J)})^T \hat \Delta _{J}   (  X_{\{\cdot,U\}}^{(i)}  - E_1 [ X_{\{\cdot,U\}}^{(1)}  ] )\right) \\
& \qquad +  \| \hat \Delta _{J}   (X_{\{\cdot,U\}}^{(i)} - E_1 [ X_{\{\cdot,U\}}^{(1)}  ]  )\|_F^2 - E_1 \| \hat \Delta _{J}   (X_{\{\cdot,U\}}^{(1)} - E_1 [ X_{\{\cdot,U\}}^{(1)}  ]  )  \|_F^2,
\end{align*}
which gives that
\begin{align*}
&\frac{K}{ n }  \sum_{i\in I}\| (\hat  L_{J} -   L ^{\star} ) ( X_{\{\cdot,U\}}^{(i)} )  \|_F^2   - E_1 \| (\hat  L_{J} -   L ^{\star} ) ( X_{\{\cdot,U\}}^{(1)} )  \|_F^2  \\
&=  2\tr\left(  (\hat \mu_Y^{(J)}   - \hat A_{J} \hat \mu_X^{(J)})^T  \hat \Delta _{J} \hat \mu_{ X}  ^{(I)} \right)  +  \tr(  \hat \Delta _{J}  (\hat \Sigma^{(I)}  - \Sigma)   \hat \Delta _{J}^T ), 
\end{align*}
with $\hat \mu_{ X}  ^{(I)} = (K /  n )  \sum_{i\in I} (  X_{\{\cdot,U\}}^{(i)}  - E_1 [ X_{\{\cdot,U\}}^{(1)}  ] )$, $\hat \Sigma^{(I)} =  (K /  n )  \sum_{i\in I} (  X_{\{\cdot,U\}}^{(i)}  - E_1 [ X_{\{\cdot,U\}}^{(1)}  ] )^{\otimes 2}$. All this together leads to
\begin{align}
\nonumber\hat \delta_{1}  = & \frac{ K }{ np } \sum_{i\in I}   \left\{  \| \epsilon_{\{\cdot,U\}}^{(i)}  \|_F^2 - E_1 [\| \epsilon_{\{\cdot,U\}}^{(1)} \|_F^2\| ]  \right\} \\
\label{eq:decomp_important}&\quad -\frac{2}{p}\tr \left(  (\hat \mu_Y^{(J)}   - \hat A_{J} \hat \mu_X^{(J)} )   \hat \mu_\epsilon^{( I)}        \right) -  \frac{2}{p}\tr \left( \hat \Delta _{J}  \hat \mu_{\epsilon X}^{( I)}        \right)  \\
\nonumber & \quad \quad +   \frac{2}{p}\tr\left(  (\hat \mu_Y^{(J)}   - \hat A_{J} \hat \mu_X^{(J)})^T \hat \Delta _{J} \hat \mu_{ X}  ^{(I)}  \right)  + \frac{1}{p} \tr(  \hat \Delta _{J}  (\hat \Sigma^{(I)} - \Sigma)   \hat \Delta _{J}^T ).   
\end{align}
The conclusion of the proof will come invoking the following results which will be proven later on.

\begin{proposition}\label{prop:l1_bound}
Under the assumptions of Proposition \ref{prop:Lasso}, we have
\begin{align*}
\max_{k=1,\ldots, p} \|  \hat \Delta_{J,k}  \|_1  =   O_P \left( \sqrt { \frac{     \log( p  ) }{n} }\right) ,
\end{align*}
where $\hat \Delta_{J,k} $ is $k$-th line of $\hat \Delta_{J}$, provided that $\lambda = C\sqrt { n\sigma^2 \log(p)}$, for some constant $C>0$.
\end{proposition}

\begin{proposition}\label{prop:other_bound}
Under the assumptions of Proposition \ref{prop:Lasso}, we have
\begin{align*}
&\max( \| \hat \mu_Y^{(J)}   \|_\infty ,\| \hat \mu_X^{(J)}   \|_\infty , \|  \hat \mu_X^{(I)}   \|_\infty , \| \hat \mu_\epsilon^{( I)}   \|_\infty,\| \hat \mu_{\epsilon X}^{( I)}   \|_\infty )    =   O_P \left( \sqrt { \frac{     \log( p  ) }{n} }\right). 
\end{align*}
\end{proposition}

Write
\begin{align*}
&\left| \frac{ K }{ np } \sum_{i\in I}   \left\{  \| \epsilon_{\{\cdot,U\}}^{(i)}  \|_F^2 - E_1 [\| \epsilon_{\{\cdot,U\}}^{(1)} \|_F^2\| ]  \right\}  \right|  =  \left| \frac{ 1 }{ p } \sum_{k=1} ^p \xi_{k,i} \right|
\leq \max_{k=1,\ldots, p}  |\xi_{k,i} |,
\end{align*}
with $\xi_{k,i}  =     (K / n ) \sum_{i\in I}  \{\|\epsilon_{\{k,U\}}^{(i)}  \|_F^2 - E_1 [\| \epsilon_{\{k,U\}}^{(1)} \|_F^2 ]\} $.
In virtue of the union bound, the previous term is then $O_P(\sqrt { \log(p) / n } )  $.  The remaining terms are bounded as follows, using that $|\tr(AB)| \leq \|A\|_\infty \|B\|_1$. We have that
\begin{align*}
\left| \tr \left(  (\hat \mu_Y^{(J)}   - \hat A_{J} \hat \mu_X^{(J)} )   \hat \mu_\epsilon^{( I)}        \right)\right| &\leq \| (\hat \mu_Y^{(J)}   - \hat A_{J} \hat \mu_X^{(J)} ) \|_\infty \| \hat \mu_\epsilon^{( I)}  \|_1\\
&\leq ( \| \hat \mu_Y^{(J)}  \|_\infty + \max _{k=1,\ldots, p} \| \hat A_{J,k}  \|_1 \| \hat \mu_X^{(J)} \|_\infty )   p T \|\hat \mu_\epsilon^{( I)}\|_\infty\\
& = O_P \left( p   \frac{\log(p) }{n} \right). 
\end{align*}
Second, we have
\begin{align*}
\left|\tr \left( \hat \Delta _{J}  \hat \mu_{\epsilon X}^{( I)}        \right)    \right | &\leq \|\hat \Delta _{J} \|_1 \|\hat \mu_{\epsilon X}^{( I)} \|_\infty = O_P \left(  p   \frac{\log(p) }{n} \right).
\end{align*}
Third, using that $\|  \hat \mu_X^{(I)}  \hat \mu_Y^{(J)T} \|_\infty \leq T \|  \hat \mu_X^{(I)} \|_\infty \|  \hat \mu_Y^{(J)T} \|_\infty $,  it holds
\begin{align*}
&\left|\tr\left(    (\hat \mu_Y^{(J)}   - \hat A_{J} \hat \mu_X^{(J)} )  ^T \hat \Delta _{J} \hat \mu_X^{(I)} \right) \right | \\
&\leq \left|\tr\left(  \hat \mu_Y^{(J)T}  \hat \Delta _{J} \hat \mu_X^{(I)} \right) \right |+
\left|\tr\left(   \hat \mu_X^{(J)T} \hat A_{J} ^T \hat \Delta _{J} \hat \mu_X^{(I)} \right) \right |\\
&= \left|\tr\left(  \hat \mu_X^{(I)}  \hat \mu_Y^{(J)T}   \hat \Delta _{J} \right) \right |+
\left|\tr\left(   \hat \mu_X^{(I)} \hat \mu_X^{(J)T} \hat A_{J} ^T \hat \Delta _{J}  \right) \right |\\
& \leq (\|  \hat \mu_X^{(I)}  \hat \mu_Y^{(J)T} \|_\infty  +
\|  \hat \mu_X^{(I)} \hat \mu_X^{(J)T} \hat A_{J} ^T\|_\infty )  \|  \hat \Delta _{J} \|_1\\
& \leq (\|  \hat \mu_X^{(I)}  \hat \mu_Y^{(J)T} \|_\infty  +
 \|  \hat \mu_X^{(I)} \hat \mu_X^{(J)T} \|_\infty \max_{k=1,\ldots, p} \| \hat A_{J,k} \|_1 )  \|  \hat \Delta _{J} \|_1\\
 & \leq T (\|  \hat \mu_X^{(I)} \|_\infty \|  \hat \mu_Y^{(J)T} \|_\infty  +
 \|  \hat \mu_X^{(I)} \|_\infty \|\hat \mu_X^{(J)T} \|_\infty \max_{k=1,\ldots, p} \| \hat A_{J,k} \|_1 )  \|  \hat \Delta _{J} \|_1\\
 & = O_P \left( p  \left( \frac{\log(p) }{n}\right) ^{3/2} \right). 
\end{align*}
Fourth, it follows that
\begin{align*}
\left| \tr(  \hat \Delta _{J}  (\hat \Sigma^{(I)} - \Sigma)   \hat \Delta _{J} ^T )   \right|\leq \|  \hat \Sigma^{(I)} - \Sigma\|_\infty \max_{k = 1,\ldots, p} \|\hat \Delta _{J,k}\|_1^2 = O_P \left( p  \left( \frac{\log(p) }{n}\right) ^{3/2} \right). 
\end{align*}
The previous bounds can be used in \eqref{eq:decomp_important} to prove the result.

\paragraph{Proof of  Proposition \ref{prop:l1_bound}} 
Without loss of generality the proof is done for $\hat A^{(lasso)} _{n}$ instead of  $\hat A_{J}$ (the only difference being the sample size $n$ instead of $n/K$). We build upon the proof of Proposition \ref{prop:Lasso}. Inject the inequality \eqref{Lasso_cone_1} in the right-hand side of \eqref{ineq:l1} to get
\begin{align*}
\|\Delta_{\{k,{S}_k^\star\}}\|_1  \leq \sqrt {  \frac{ 6 |{S}_k^\star| \lambda }{ n \gamma^\star }    \|\Delta_{\{k,{S}_k^\star\}}\| _1}.
\end{align*}
Because $\Delta_{\{k,{S}_k^\star\}} \in \mathcal C (3,S^\star_k ) $, it follows that
\begin{align*}
\|\Delta_{\{k,\cdot\}} \|_1 \leq 4 \|\Delta_{\{k,{S}_k^\star\}} \|_1  \leq  24 \frac{|{S}_k^\star| \lambda }{ n \gamma^\star} = C' \sqrt {\frac{ \sigma^ 2       \log( p  ) }{n}} ,
\end{align*}
for some $C'>0$ that does not depend on $(n,p)$, with probability at least $1-\delta$. 

\qed

\paragraph{Proof of  Proposition \ref{prop:other_bound}} The proof follows from the application of Lemma \ref{lemma:basic_stuff} and Proposition \ref{lemma:unif_epsilon}.

\subsection{Proof of Corollary \ref{cor:regime_switch2}}
Introduce the notation $R_t = R (L^\star _t)$. Since $\hat t = t^\star$ if $ \hat R _{t^\star} < \hat R_t$ for all $t \neq t^\star$, it suffices to show that
$ P (\hat R _{t^\star} - \min _{ t\neq t^\star} \hat R_t < 0  ) \to 1$. We have that
\begin{align*}
\hat R _{t^\star} - \min_{ t\neq t^\star} \hat R _ t  \leq R_{t^\star}  - \min_{t\neq t^\star} R_t  + \epsilon_n.
\end{align*}
with $\epsilon _n = 2 \max_{t=1,\ldots, T} |  \hat R _ t - R_t |$ and from Proposition \ref{prop:regime_switch1} $ \epsilon _n = o_P (1)$. But for $p$ large enough, $R_{t^\star} - \min_{t\neq t^\star} R_t < - \epsilon <0$. As a consequence, 
$ P (\hat R _{t^\star} - \min _{ t\neq t^\star} \hat R_t < 0  ) \geq  P (\epsilon_ n < \epsilon) \to 1$.
\qed

\section{Intermediate results}

Define the standardized predictors $Z = \Sigma^{-1/2} ( X  - E(X) ) $, and for any $i=1,\ldots, n$, $Z^{(i)} = \Sigma^{-1/2} ( X^{(i)}  - E(X) ) $. 

\begin{lemma}\label{lemma:basic_stuff}
Suppose that Assumptions \ref{ass_1}, \ref{ass_2}, \ref{ass_4} and \ref{ass_3} are fulfilled.  It holds that
\begin{align*}
\|   \overline{Y}^n - E(Y) \|_F  = O_P \left( \sqrt {p/n}  \right) .
\end{align*}
Moreover, for any $A\in \mathbb R^{ q\times p}$, it holds
\begin{align*}
\|  A \overline {Z}^n   \|_F  = O_P \left( \sqrt {\|A\|_F^2 /n}  \right) .
\end{align*}

\end{lemma}

\begin{proof}
If $M_i$ are i.i.d. centered random matrices, we have that $ E (\| \sum_{i=1}^n M_i \|_F^2) = n E( \| M_1 \|_F^2 )$.
It follows  that
\begin{align*}
E (\|   \overline{Y}^n - E(Y) \|_F^2) = n^{-1}  E[ \| {Y} - E (Y)\|_F^2 ]  \leq n^{-1} Tp \max_{t\in \mathcal T,\, k\in \mathcal{S}} \var(Y_{k,t}) . 
\end{align*}
The conclusion comes from using Chebyshev's inequality and the Assumption  \ref{ass_3}. For the second assertion, because $E[  Z   Z ^T    ] = I _ p $, we have
\begin{align*}
E( \|  A \overline {Z}^n  \|_F^2 )   = n^{-1} \tr( A E[  Z   Z ^T    ]A^T  ) = n^{-1}\| A \|_F^2    .
\end{align*}
\end{proof}

\begin{lemma}\label{lemma:residuals_frob}
Suppose that Assumptions \ref{ass_1}, \ref{ass_2}, \ref{ass_4} and \ref{ass_3} are fulfilled.  We have
\begin{align*}
&\|   \sum_{i=1}^n \epsilon^{(i)}  \|_F  = O_P \left( \sqrt {n \sigma^2 p}  \right) \\
&\|   \sum_{i=1}^n  (\epsilon^{(i)} - \overline{\epsilon}^n  ) ( Z^{(i)} -\overline{Z}^n )^T  \|_F  = O_P \left( \sqrt { n \sigma^2 p^2}   \right) .
\end{align*}
\end{lemma}

\begin{proof}
The first statement follows from
\begin{align*}
E[\|   \sum_{i=1}^n \epsilon^{(i)}  \|_F^2] = n  E \| \epsilon \|_F^2=   n  \sum_{t= 1}^T \sum_{k=1}^p   E [ \epsilon_{t,k} ^2]\leq n Tp \sigma ^2.
\end{align*}
For the second statement, start by noting that
\begin{align}\label{eq:decomp_covariance}
  \sum_{i=1}^n  (\epsilon^{(i)} - \overline{\epsilon}^n  ) ( Z^{(i)} -\overline{Z}^n )^T =   \sum_{i=1}^n  \epsilon^{(i)}   Z^{(i)T} -  n  \overline{\epsilon}^n  (\overline{Z}^{n})^T.
\end{align}
Then use the triangle inequality to get
\begin{align*}
\| \sum_{i=1}^n  (\epsilon^{(i)} - \overline{\epsilon}^n  ) ( Z^{(i)} -\overline{Z}^n )^T  \|_F&\leq \|   \sum_{i=1}^n  \epsilon^{(i)}  Z^{(i)T}  \|_F +  n \|  \overline{\epsilon}^n  (\overline{Z}^{n})^T\|_F\\
&\leq \|   \sum_{i=1}^n  \epsilon^{(i)}  Z ^{(i)T}  \|_F + \| \overline {Z}^n \| _F  \|\sum_{i=1}^n  \epsilon^{(i)} \|_F.
\end{align*}	
Using Lemma \ref{lemma:basic_stuff} and the first statement, we find that $\| \overline {Z} \| _F  \|\sum_{i=1}^n  \epsilon^{(i)} \|_F = O_P( \sigma p  ) $, Hence, showing that 
\begin{align*}
\|   \sum_{i=1}^n  \epsilon^{(i)}  Z ^{(i)T}  \|_F = O_P (\sqrt n\sigma p ),  
\end{align*}
will conclude the proof. Using Assumption \ref{ass_2} and Proposition \ref{prop:uniq_min_homo}, we find that $ E [  \tr(  \epsilon^{(i)}  Z^{(i)T} Z^{(j)} \epsilon^{(j)T}   )  ]  = 0$ for all $i\neq j$. It follows that
\begin{align*}
 E [ \| \sum_{i=1}^n  \epsilon^{(i)} Z^{(i)T}  \|_F^2  ] =    n  E [  \|  \epsilon^{(1)}  Z^{(1)T}  \|_F^2    ] .
\end{align*}
Using the triangle inequality and Jensen's inequality, we get
\begin{align*}
 E [ \| \sum_{i=1}^n  \epsilon^{(i)} Z ^{(i)T}  \|_F^2  ]  & \leq   n  E \left[   \left(\sum_{t=1}^T \|  \epsilon_{\{\cdot,t\}}^{(1)}  Z_{\{\cdot,t\}} ^{(1)T}  \|_F \right)^2    \right] \\
& \leq n T \sum_{t=1}^T   E [   \|  \epsilon_{\{\cdot,t\}}^{(1)}  Z_{\{\cdot,t\}}^{(1)T}  \|_F ^2    ]\\
& =  n T \sum_{t=1}^T  E [  \tr(  \epsilon_{\{\cdot,t\}}^{(1)} \| Z_{\{\cdot,t\}} ^{(1)}\|_2^2 \epsilon_t^{(1)T}   )  ] \\
&  =   n T \sum_{t=1}^T  E [  \| Z_{\{\cdot,t\}}^{(1)} \|_2^2 \| \epsilon_{\{\cdot,t\}}^{(1)}\|_2^2     ]  \\
&\leq nT p \sigma^2  \sum_{t=1}^T  E [  \| Z_{\{\cdot,t\}}^{(1)} \|_2^2] = nTp^2\sigma^2.
\end{align*}
The last inequality is a consequence of the definition of $\sigma^2$.
\end{proof}

Recall that
\begin{align*}
&\hat \Sigma =n^{-1} \sum_{i=1} ^n  (X^{(i)} - \overline{X}^n)  (X^{(i)} - \overline{X}^n) ^T, \\
&\hat \Pi =  \Sigma^{1/2} \hat \Sigma \Sigma^{1/2}.
\end{align*}

\begin{lemma}\label{lemma:eigenvalues}
Suppose that Assumptions \ref{ass_1} and \ref{ass_2} are fulfilled and that
\begin{align*}
B \geq \tr(  (X - E(X))^T \Sigma^{-1}  (X- E(X)) ) .
\end{align*}
Then, for any $(n,\delta)$ such that $n\geq 4 B  \log( 4 p T / \delta ) $, it holds, with probability $1-\delta$: 
\begin{align*}
\|\hat \Pi  - I _ {p} \|  \leq   4 \sqrt{ n^{-1}B  \log( 4 p  T / \delta ) } .
\end{align*}
\end{lemma}

\begin{proof}
Remark that
\begin{align*}
 \hat \Pi  =  \tilde \Pi  -   \overline {Z }^n \, \overline {Z }^{nT},   
\end{align*}
with $  \tilde \Pi = n^{-1} \sum_{i=1} ^n  (X - EX )  (X - EX ) ^T$. Apply the triangle inequality to get
\begin{align*}
  \| \hat \Pi - I _{p} \|   \leq \|  \tilde \Pi -I _{p}  \|   +   \| \overline {Z }^n \, \overline {Z }^{nT}\| = \|  \tilde \Pi -I_{ p }  \|   +   \| \overline {Z }^n \|_2^2 .
\end{align*}
We shall apply Lemma \ref{lemma:matrix_bernstein} to control each of the two terms in the previous upper bound. First consider the case where $S^{(i)} = (Z^{(i)}   Z^{(i)T}   - I _{p}) / n $. We need to specify the value for $L$ and $v $ that we can use.
Note that 
\begin{align}\label{useful_control_eigen}
\|Z^{(i)}   Z^{(i)T } \|  =\|\sum _{ t = 1 }^T Z^{(i)} _ {\{\cdot, t\}}  Z^{(i)T}_ {\{\cdot,t\}}  \| \leq \sum _{ t = 1 }^T \|Z_{\{\cdot,t\}} ^{(i)} Z_{\{\cdot,t\}} ^{(i)T} \|\leq  \sum _{ t = 1 }^T \| Z_{\{\cdot,t\}} ^{(i)} \|_2^2 \leq   B.
\end{align}
Using the triangle inequality and Jensen inequality, we have, for any $i=1,\ldots, n $,
\begin{align*}
\|S_i\| &\leq   n^{-1} (\|Z  ^{(i)} Z  ^{(i)T }\|  +  E [\| Z  Z  ^{T } \|] )\leq 2 Bn^{-1}.
\end{align*}
Consequently we can take $L = 2B n^{-1}$ for the value of $L$. Moreover, we have 
\begin{align*}
\| \sum_{i=1} ^n E [ ( Z^{(i)} Z^{(i)T} -I_p )^T  ( Z^{(i)} Z^{(i)T} -I_p )  ]  \| &=  \| \sum_{i=1} ^n (E [ ( Z^{(i)} Z^{(i)T}  Z^{(i)} Z^{(i)T}]  -  I_p)    \|\\
& = n  \| E [ Z Z^T Z Z^T]  -  I_p\|\\
&\leq  n\| E [  Z Z^{T}  Z Z^{T}]  \|\\
&= n \sup_{\|u\| = 1}  E [ u^T  Z Z^{T}  Z Z^{T} u ]\\
&= n \sup_{\|u\| = 1}  E [\|  Z Z^{T} u\|^2_2 ]. 
\end{align*}
From the classic inequality $\| A^{1/2} u\|_2^2 \leq \| A\| \|u\|_2^2 $, we deduce that $ \| A u\|_2^2 \leq \| A\| \| A^{1/2} u\|_2^2  $. Applying this with $ A =  Z Z ^{T}$ and using \eqref{useful_control_eigen}, we obtain
\begin{align*}
 \|  Z Z^{T}  u \|_2^2\leq \| Z Z^{T}  \| \|  (Z  Z^{T})^{1/2} u \|_2^2\leq B \|  (Z  Z^{ T})^{1/2} u \|_2^2 .
\end{align*}
Taking the expectation in the previous inequality, we get $E[\|  Z Z^{T}  u \|_2^2 ] \leq B \|u\|_2^2$. Hence
\begin{align*}
\| \sum_{i=1} ^n E [ ( Z^{(i)} Z^{(i)T} -I_p )^T  ( Z^{(i)} Z^{(i)T} -I_p )  ]  \| \leq nB.
\end{align*}
Hence we can take $n^{-1} B$ for the value of $v$. Lemma \ref{lemma:matrix_bernstein} gives that
\begin{align*}
\mathbb P \left( \|  \tilde \Pi -I _{ p }  \|   > t \right) \leq  2 p    \exp \left(\frac{-t^{2} }{ 2n^{-1}B (1 +  2 t / 3) }\right).
\end{align*}
Consequently, with probability $1-\delta / 2$, it holds that
\begin{align*}
 \|  \tilde \Pi -I _{p}  \| \leq  \sqrt{  4n^{-1}B  \log( 4 p  / \delta ) },
\end{align*}
provided that $ {  4n^{-1}B   \log( 4 p  / \delta ) } \leq 9 / 4$. Now we apply Lemma \ref{lemma:matrix_bernstein} with $ S^{(i)} = Z^{(i)} /n $ to provide a bound on $\| \overline {Z }^n \|$. By \eqref{useful_control_eigen}, we have that
\begin{align*}
&\|Z \| = \|Z Z^T  \|^{1/2}\leq B^{1/2},\\
& \max\left\{ \|   E [ Z Z ^{T}   ]\|, \|   E [ Z^{T} Z   ]\| \right\} = \max\left\{  \|I_{p}\| , p \right\} = p\leq B. 
\end{align*}
Using $p + T\leq 2pT$, it follows that
\begin{align*}
\mathbb P\left(  \|\overline {Z }^n \|>t \right)   \leq 2 p T  \exp \left(\frac{-t^{2} }{ 2(B n^{-1} +   t  B^{1/2} n^{-1} / 3) }\right).
\end{align*}
By taking $t = \sqrt{  4n^{-1}B  \log( 4 pT / \delta ) } $ we get that $\|\overline {Z }^n \|> t$ with probability smaller than
 \begin{align*}
  2 p T \exp \left(\frac{- ( \log( 4 p T/ \delta ) )  }{  ( 1 +   \sqrt {4n^{-1} \log( 4 p T / \delta )  } / 3) }\right)\leq   \delta/2.
 \end{align*}
provided that $ 4n^{-1} \log( 4 p T / \delta ) \leq 9$. Hence we have shown that with probability $1-\delta$,
 \begin{align*}
 \|\Delta\| + \|\overline {Z }^n \|&\leq \sqrt{  4n^{-1}B   \log( 4 p  / \delta ) } + {  4n^{-1}B   \log( 4 p T / \delta ) }\\
 &\leq \sqrt{  4n^{-1}B \log( 4 p T / \delta ) } ( 1 +   \sqrt{ 4n^{-1}B   \log( 4 p T / \delta ) } ).
 \end{align*}
 Use that $4n^{-1}B   \log( 4 p T / \delta )\leq 1$ to conclude.
\end{proof}

\begin{lemma}\label{lemma:unif_epsilon}
Suppose that Assumptions \ref{ass_1} and \ref{ass_2} are fulfilled. Let $ U $ be such that $ \max _{k\in \mathcal S, \, t\in \mathcal T} | X_{k,t} - E(X_{k,t} )  |\leq U$  and $\max _{k\in \mathcal S, \, t\in \mathcal T} | \epsilon_{k,t}    |\leq U$, almost surely. If $T\leq p$ and  $   n \geq 4 (U^2 / \sigma^2) \log( 6p^2 / \delta  )  $,  we have with probability $1-\delta$:
\begin{align*}
\left\|\sum_{i=1}^n (\epsilon  ^{(i)} - \overline{\epsilon}^n )  (X^{(i)} - \overline{X}^n )^T  \right\|_\infty \leq 
4\sqrt{  T^2 U^2 \sigma^ 2 n \log( 6 p^2 / \delta  )}.
\end{align*}
\end{lemma}

\begin{proof}
We have
\begin{align*}
&\left\|\sum_{i=1}^n (\epsilon ^{(i)} - \overline{\epsilon}^n )  (X^{(i)} - \overline{X}^n )^T  \right\|_\infty\\
& = \left\|\sum_{i=1}^n \epsilon ^{(i)}    (X^{(i)} - \overline{X}^n )^T  \right\|_\infty\\
&\leq \left\| \sum_{i=1}^n  \epsilon^{(i)}   (X^{(i)} - E(X) )^T \right\|_\infty  +  n \left\|  \overline{\epsilon}^n  (\overline{X}^{n} - E(X) )^T\right\|_\infty\\
&\leq  \left\| \sum_{i=1}^n  \epsilon^{(i)}  (X^{(i)} - E(X) )^T \right\|_\infty  +  n \left\|  \sum_{t=1}^T \overline{\epsilon_{\{\cdot, t\}}}^n  (\overline{X_{\{\cdot, t\}}}^{n}- E(X_{\{\cdot, t\}}))^T\right\|_\infty\\
&\leq  \left\| \sum_{i=1}^n  \epsilon^{(i)}   (X^{(i)} - E(X) )^T \right\|_\infty  +  nT \max_{t=1,\ldots,T}  \left\|  \overline{\epsilon_{\{\cdot, t\}}}^n  (\overline{X_{\{\cdot, t\}}}^{n} - E(X_{\{\cdot, t\}}) )^T\right\|_\infty\\
&\leq \left\| \sum_{i=1}^n  \epsilon^{(i)}   (X^{(i)} - E(X) )^T \right\|_\infty  +  nT \max_{t=1,\ldots,T}  \left\|  \overline{\epsilon_{\{\cdot, t\}}}^n \|_\infty \| \overline{X_{\{\cdot, t\}}}^{n} - E(X_{\{\cdot, t\}})  \right\|_\infty\\
&\leq \left\| \sum_{i=1}^n  \epsilon^{(i)}   (X^{(i)} - E(X) )^T \right\|_\infty  +  nT  \|  \overline{\epsilon }^n \|_\infty \left\| \overline{X}^{n} - E(X)  \right\|_\infty.
\end{align*}
For the first term, we apply Lemma \ref{lemma:bernstein} with $A^{(i)} =   \epsilon^{(i)}   (X^{(i)} - E(X) )^T $ noting that, by Jensen inequality,
\begin{align*}
\var\left( (\epsilon^{(1)}   (X^{(1)} - E(X) )^T )_{k,\ell} \right)&= E \left[ (\epsilon^{(1)}   (X^{(1)} - E(X) )^T )_{k,\ell} ^2 \right] \\
&= E\left[ ( \sum_{t=1} ^T \epsilon_{k,t}  (X_{\ell,t} - EX_{\ell,t}  ) )^2 \right]  \\
&\leq  T\sum_{t=1} ^T  E( (  \epsilon_{k,t}  (X_{\ell,t} - EX_{\ell,t}  ) ) ^2 )  \\
&\leq T^2 \sigma^2 \max_{\ell \in S ,t\in \mathcal T} E( (X_{\ell,t} - EX_{\ell,t}   ) ^2 )  \\
&\leq T^2 U^2 \sigma^2.
\end{align*}
we may take $v = T^2 U^2 \sigma^2$ and $ c= T U^2$. Then using that $9 T^2 U^2 \sigma^2   n \geq 4 T^2U^4 \log( 2p^2 / (\delta / 3) ) $, we have with probability $1-\delta/3$,
 \begin{align*}
 \left\| \sum_{i=1}^n  \epsilon^{(i)}   (X^{(i)} - E(X) )^T \right\|_\infty\leq \sqrt{ 4 T^2 U^2 \sigma^ 2 n \log( 6 p^2 / \delta  )} .
 \end{align*}
 Applying again Lemma \ref{lemma:bernstein}, to the sequence $(\epsilon^{(i)})_{i=1,\ldots, n}$, taking $v = \sigma^2$ and $ c = U$ and using $9\sigma^2 n\geq 4 U ^2\log( 2p^2 / (\delta / 3) )  \geq 4 U ^2\log( 2p T / (\delta / 3) ) $, we deduce that with probability $1-\delta/3$,
\begin{align*}
\|  \overline{\epsilon }^n \|_\infty \leq \sqrt{ 4\sigma^ 2  \log( 6 p T / \delta  ) / n } .
\end{align*} 
Applying again Lemma \ref{lemma:bernstein}, to the sequence $(X^{(i)} - EX )_{i=1,\ldots, n}$, taking $v = U^2$ and $ c =U$ and using $9  n\geq 4 (U^2/\sigma^2) \log( 2p^2 / (\delta / 3) ) \geq 4  \log( 2p T / (\delta / 3) ) $, we deduce that with probability $1-\delta/3$
\begin{align*}
\| \overline{X}^{n} - E(X)  \|_\infty \leq \sqrt{ 4 U^ 2  \log( 6 p T / \delta  ) / n}.
\end{align*}
 All this together with the fact that $T\leq p$, we get
 \begin{align*}
&\left\|\sum_{i=1}^n (\epsilon  ^{(i)} - \overline{\epsilon}^n )  (X^{(i)} - \overline{X}^n )^T  \right\|_\infty \\
&\leq 
\sqrt{ 4 T^2 U^2 \sigma^ 2 n \log( 6 p^2 / \delta  )} + T \sqrt{ 4\sigma^ 2  \log( 6 p^2 / \delta  )  4 U^ 2  \log( 6 p^2 / \delta  ) }  \\
&= \sqrt{ 4 T^2 U^2 \sigma^ 2 n \log( 6 p^2 / \delta  )} \left( 1 + \sqrt{   4   \log( 6 p^2 / \delta  ) / n  } \right),
\end{align*}
and the conclusion follows from $ n\geq   4 (U^2/\sigma^2)    \log( 6 p^2 / \delta  )  \geq   4   \log( 6 p^2 / \delta  )   $.
\end{proof}

\begin{lemma}\label{lemma:unif_sigma}
Suppose that Assumptions \ref{ass_1}, \ref{ass_2}  and \ref{ass_4} are fulfilled. If $T\leq p$ and $n \geq  \log(4p^2 / \delta )$   we have with probability $1-\delta$:
\begin{align*}
\| \hat \Sigma - \Sigma \|_\infty \leq & 8 T U^2 \sqrt{  \log( 4p^2/ \delta )  /n },
\end{align*}
where $U$ is such that $ \max _{k\in S, \, t\in \mathcal T} | X_{k,t} - E(X_{k,t} )  |\leq U$ almost surely.
\end{lemma}

\begin{proof}
We have
\begin{align*}
\hat \Sigma - \Sigma = n^{-1}  \sum_{i=1} ^n \left\{(X^{(i)} - E(X) )  (X^{(i)}  - E(X))^T   - \Sigma \right\} - (E(X) - \overline{X}^n)  (E(X) - \overline{X}^n) ^T.
\end{align*}
In the following, we bound each of the two previous terms under events of probability $1-\delta/2$, respectively. Apply Lemma \ref{lemma:bernstein} with $A^{(i)} = (X^{(i)} - E(X) )  (X^{(i)}  - E(X))^T   - \Sigma$. Note that
\begin{align*}
|A^{(1)}_{k,\ell}| &\leq 2 | (X^{(1)}_{\{k,\cdot\}} - E(X_{\{k,\cdot\}}) )^T  (X_{\{\ell,\cdot\}}^{(1)}  - E(X_{\{\ell,\cdot\}}))| \\
& = 2 \left| \sum_{t=1}^T (X^{(1)}_{k,t} - E(X_{k,t} ) )  (X_{\ell,t}^{(1)}  - E(X_{\ell,t})) \right|\\
&\leq 2 T U^2.
\end{align*}
and
\begin{align*}
\var (A^{(1)}_{k,\ell} ) & \leq E\left[ \left( (X^{(1)}_{\{k,\cdot\}} - E(X_{\{k,\cdot\}} ) )^T  (X_{\{\ell,\cdot\}}^{(1)}  - E(X_{\{\ell,\cdot\}}))\right)^2 \right]   \leq  (2 T U^2)^2.
\end{align*}
 Hence we apply Lemma \ref{lemma:bernstein} with $v=(2 T U^2)^2$ and $c = 2 T U^2$. We obtain, because $  9 n \geq  4 \log(4p^2 / \delta ) $, with probability $1-\delta/2$,
\begin{align*}
\left\| n^{-1}  \sum_{i=1} ^n \left\{(X^{(i)} - E(X) )  (X^{(i)}  - E(X))^T   - \Sigma \right\}   \right\|_\infty \leq 4 T U^2 \sqrt{  \log( 2p^2/ (\delta/2) )  /n }.
\end{align*}
Note that
\begin{align*}
((E(X) - \overline{X}^n)  (E(X) - \overline{X}^n) ^T)_{k,\ell}  = \sum_{t=1} ^T (E(X_{k,t}) - \overline{X_{k,t}}^n)  (E(X_{\ell,t}) - \overline{X_{\ell,t}}^n)  .  
\end{align*}
It follows that
\begin{align*}
\left| ((E(X) - \overline{X}^n)  (E(X) - \overline{X}^n) ^T)_{k,\ell} \right| \leq T \max_{t=1,\ldots, T} \max_{k=1,\ldots, p}   | E(X_{k,t}) - \overline{X_{k,t}}^n|^2   .  
\end{align*}
Applying Lemma \ref{lemma:bernstein} with $A^{(i)} = X^{(i)}  - E(X) $, $c = U $, and $v = U^2$ we get, because $9  n \geq 4 \log(2pT / (\delta/2)$, with probability $1-\delta/2$,
\begin{align*}
\max_{t=1,\ldots, T} \max_{k=1,\ldots, p}   | E(X_{k,t}) - \overline{X_{k,t}}^n|\leq  \sqrt{ 4  U^2 \log(2pT/(\delta/2)) / n }.
\end{align*}
Hence we have shown that, with probability $1-\delta$,
\begin{align*}
\| \hat \Sigma - \Sigma \|_\infty &\leq 4 T U^2 \sqrt{  \log( 4p^2/ \delta )  /n } +   4 T  U^2 \log(4pT/\delta) / n \\
&\leq 4T U^2 \sqrt{  \log( 4p^2/ \delta )  /n } ( 1 +  \sqrt{  \log(4p^2/\delta) / n} ).
\end{align*}
Invoking that $n \geq  \log(4p^2/\delta)$ we obtain the result.
\end{proof}

\section{Auxiliary results}\label{sec:aux}

\begin{lemma}[Bernstein inequality]\label{lemma:bernstein}
Suppose that $(A^{(i)})_{i\geq 1}$ is an iid sequence of centered random vectors valued in $\mathbb R^{p\times T}$. Suppose that $\|A^{(1)}\|_\infty \leq c $ a.s. and that $\var(A^{(1)}_{k,\ell}) \leq  v $.
For any $\delta \in (0,1) $, $n\geq 1$ such that $  9  {vn}  \geq     4  c^2  \log( 2 pT / \delta )   $, we have with probability $1-\delta$:
\begin{align*}
\left\|\sum_{i=1}^n A^{(i)} \right\|_\infty\leq  \sqrt { 4 n v \log( 2 pT / \delta )  }.
\end{align*}
\end{lemma}

\begin{proof}
By Bernstein inequality, it follows that
\begin{align*}
\mathbb P \left( \left|\sum_{i=1}^n  A^{(i)}_{k,t} \right| >x\right) &\leq 2 \exp\left(  - \frac{x^2}{2 ( v n +  c x / 3)}  \right) .
\end{align*}
Choosing $x= \sqrt { 4 n v \log( 2 pT / \delta )  } $ and because $v n \geq    c x / 3$, we get that
\begin{align*}
\mathbb P \left( \left|\sum_{i=1}^n A^{(i)}_{k,t} \right| >x \right) &\leq \delta / (pT).
\end{align*}
Use the union bound to get that, for any $t>0$,
\begin{align*}
\mathbb P \left( \left\|\sum_{i=1}^n A^{(i)}  \right\|_\infty > x \right) \leq \sum_{1\leq k,t\leq p} \mathbb P \left( \left|\sum_{i=1}^n A^{(i)}_{k,t}\right| >x \right) \leq \delta .
\end{align*}

\end{proof}

\begin{lemma}[Matrix Bernstein inequality]
    \label{lemma:matrix_bernstein}
    Let $S^{(1)}, \ldots, S^{(n)}$ be independent, centered random matrices with common dimension $d_{1} \times d_{2},$ and assume that each one is uniformly bounded in spectral norm, i.e.,
    $$ E[S^{(i)}]=0 \text { and }\left\|S^{(i)}\right\| \leq L \quad \text { for each } i=1, \ldots, n. $$
    Let $v>0$ be such that
    $$\max \left\{\left\|\sum_{i=1}^{n} E\left({S}^{(i)} {S}^{(i)T}\right)\right\|,\left\| \sum_{i=1}^{n} E\left({S}^{(i)T} {S}^{(i)}\right)\right\|\right\}\leq v.$$
    Then, for all  $t \geq 0,$
    $$ \mathbb P \left(  \|\sum\limits_{i=1}^n S^{(i)}\| \geq t \right) \leq  \left(d_{1}+d_{2}\right) \exp \left(\frac{-t^{2} / 2}{v+L t / 3}\right). $$  
\end{lemma}

\end{appendices}

\end{document}